\documentclass[11pt]{article}
\usepackage{amssymb}
\usepackage{amsfonts}
\usepackage{mathrsfs}
\usepackage{amssymb,amsfonts,amsmath}
\usepackage{graphics}
\usepackage{graphicx}
\usepackage{subfigure}
\usepackage{latexsym}
\usepackage{esint}
\usepackage{epstopdf}
\usepackage{graphicx,float}
\usepackage{times}
\usepackage{graphicx,multicol,enumerate,subfigure,amsmath}
\usepackage{color}
\usepackage{fullpage}
\usepackage{setspace}
\usepackage[english]{babel}
\usepackage[version=3]{mhchem}
\usepackage{amsmath, amssymb, amsthm}
\usepackage{verbatim, latexsym}
\usepackage{colortbl}
\usepackage{multirow}
\newcommand{\diam}[1]{\mathrm{diam}(#1)}
\newcommand{\bs}{\boldsymbol}
\newcommand{\ext}{C_{\rm{ext}}(D,s)}

\newcommand{\prnt}[1]{\left( #1 \right)}

\newcommand{\norm}[1]{\left\|#1\right\|}
\newcommand{\normHsemi}[2]{\left|#1\right|_{H^{1}\prnt{#2}}}
\newcommand{\normHNg}[2]{\norm{#1}_{H^{-1}\prnt{#2}}}

\newcommand{\normL}[2]{\norm{#1}_{L^2\prnt{#2}}}
\newcommand{\normHp}[3]{\norm{#1}_{H^{#2}\prnt{#3}}}

\newcommand{\normI}[2]{\norm{#1}_{C\prnt{#2}}}
\newcommand{\normLH}[3]{\norm{#1}_{L^2(\Omega,H^{#2}\prnt{#3})}}

\newcommand{\mc}[1]{\mathcal{#1}}
\newcommand{\CI}{C_{I_h}}
\newtheorem{theorem}{Theorem}[section]

\newtheorem{assumption}{Assumption}[section]
\newtheorem{remark}{Remark}[section]
\newtheorem{lemma}{Lemma}[section]

\newtheorem{proposition}{Proposition}[section]
\numberwithin{equation}{section}
\usepackage{soul}
\setstcolor{red}
\newcommand{\noteLi}[1]{{\color{red}{#1}}}

\newcommand{\dy}{\boldsymbol{\rho}\mathrm{d}{y}}

\title{On the Numerical Approximation of the Karhunen-Lo\`{e}ve Expansion for Lognormal Random Fields}
\author{Michael Griebel\thanks{Institut f\"ur Numerische Simulation,
            Universit\"at Bonn,
            Endenicher Allee 19B, 53115 Bonn, Germany.
            E-mail: griebel@ins.uni-bonn.de
             }\;\,\thanks{Fraunhofer SCAI, Schloss Birlinghoven, 53754 Sankt Augustin, Germany} \; and Guanglian Li\thanks{Bernoulli Institute, University of Groningen, Nijenborgh 9, 9747AG, Groningen, The Netherlands. E-mail: guanglian.li@rug.nl, lotusli0707@gmail.com.
             }
}
\date{}
\begin{document}
\maketitle
\begin{abstract}
The Karhunen-Lo\`{e}ve (KL) expansion is a popular method for approximating random fields by transforming an
infinite-dimensional stochastic domain into a finite-dimensional parameter space. Its
numerical approximation is of central importance to the study of PDEs with random coefficients. In this work, we analyze
the approximation error of the Karhunen-Lo\`eve expansion for lognormal random fields. We derive error
estimates that allow the
% This work is a continuation of \cite{griebel2017decay}, and we are mainly concerned with
optimal balancing of the truncation error of the expansion, the Quasi Monte-Carlo error for sampling in the stochastic domain and the numerical
approximation error in the physical domain. The estimate is given in the number $M$ of terms maintained in the KL expansion, in the number of sampling points $N$, and in the discretization mesh size
$h$ in the physical domain employed in the numerical solution of the eigenvalue problems during the expansion.
% the Karhunen-Lo\`{e}ve (KL) expansion. This
%problem arises from partial differential equations with random coefficient where the parameterization strategy is a crucial step
%in the process of numerical approximation to transform  an infinite-dimensional stochastic domain into a finite-dimensional parameter space.
%We consider the approximation of the popular lognormal random field. % in this paper.
The result is used to quantify the error in PDEs
with random coefficients. We complete the theoretical analysis with numerical experiments in one
and multiple stochastic dimensions.\\
%also derive the error estimates and will result in an accurate and efficient parameterized problem .\\
{\bf Keywords}: Karhunen-Lo\`eve expansion, eigenvalue decay, approximation of bivariate functions, error estimates, lognormal random field.
\end{abstract}

\section{Introduction}\label{sec:intro}
Partial differential equations (PDEs) with random coefficient have been widely employed to describe applications
that are affected by a certain amount of uncertainty arising from imperfect/insufficient information about the problem,
e.g., in the input data. The range of applications is broad and diverse and includes, e.g., oil field modelling, quantum mechanics
and finance \cite{B-Brown.Kuo.Sloan:2017, Griebel.Hamaekers.Chinnamsetty:2016, sun2017new}. The dimension of the random
coefficient can be huge or even infinite, which poses enormous computational challenge. To reduce its dimensionality, one can parameterize
the random coefficient by means of the Karhunen-Lo\`{e}ve (KL) expansion or the polynomial chaos (PC) expansion
\cite{ghanem2003stochastic, Schwab:2006:KAR:1167051.1167057}, which greatly facilitates the subsequent numerical treatment,
e.g., by the stochastic Galerkin method or the stochastic collocation method. Alternatively, one may expand the random field with respect
to the hierarchical Faber basis or some wavelet type basis; see \cite{Bachmayr.Cohen.Dinh.Schwab:2017, cohen2010convergence}
for details. In this paper, we will focus on the KL expansion, which is known to be optimal in the sense of the mean square error.

To formulate the problem, let $D\subset\mathbb{R}^d$ be an open bounded domain with a strong local Lipschitz boundary and let $(\widetilde{\Omega},\Sigma, \mathcal{P})$ be
a complete separable probability space with $\sigma$-field $\Sigma\subset 2^{\widetilde{\Omega}}$ and probability measure $\mathcal{P}$. We
will denote $\Omega:= (\widetilde{\Omega},\Sigma, \mathcal{P})$ for notational simplicity. Now, we consider a stochastic {field} $\kappa(y,x)
\in L^{\infty}(\Omega, L^2(D))$ with its logarithm being a centered Gaussian field. The lognormal
random field is frequently used in stochastic PDEs as a random diffusion coefficient.
%The main goal of this work is to provide quantitative estimates on the number of terms required in approximating
%$\kappa(y,x)$ within a prescribed accuracy, when the input data vary slowly in the physical domain and the
%correlation length is comparable to the size of the domain.
%However, we shall intend not to discuss error polluted by the data or the modeling.

In practical computation, its numerical approximation usually proceeds in three steps. In the first step, the centered random field $\log \kappa(y,x)$ is approximated by its $M$-term KL expansion for some $M\in \mathbb{N}_{+}$. The truncation error relies on the regularity of the bivariate function $\log \kappa(y,x)$ in the physical variable $x$, see \cite{griebel2017decay} for details. In the second step, the covariance
function $R(x,x')$ of the centered Gaussian random field $\log \kappa(y,x)$ is approximated via a sampling method. By its very definition, the covariance function involves an integral
over the stochastic domain $\Omega$, which is often of very high dimensional. For its approximation, various quadrature-type sampling methods, e.g., Monte-Carlo methods, (Quasi) Monte-Carlo (QMC) methods and sparse grids
\cite{Bungartz.Griebel:2004, MR3038697,DuTeUl2015} can be applied, say with $N$ sampling points. These methods essentially require
boundedness of the variation, the first or higher mixed derivatives of $\log\kappa(\cdot,x)$ for fixed $x\in D$ and then yield a corresponding order of convergence. In this paper, we focus on the QMC method, which has only a low regularity
requirement on $\Omega$, namely that the first mixed derivative of $\log\kappa(\cdot,x)$ is bounded. %Note that one may choose other sampling methods according to the specific problem structure.
The outcome of this second step is a function $R_N(x,x')\in L^2(D\times D)$ that approximates the covariance function $R(x,x')$. The associated
self-adjoint operators are denoted as $\mc{R}_N$ and $\mc{R}$, respectively. Note that $\mc{R}_N$ is a finite rank operator with rank
not larger than $N$. We shall prove in Proposition \ref{prop:recall} that only the first $\lfloor  N^{\frac{1}{2s/d+1}}\rfloor$ terms
in the KL expansion of $\mc{R}_N$ are relevant to approximate the spectrum of $R(x,x')$. Here, the nonnegative parameter $s$ denotes the regularity
of the bivariate function $\log \kappa(y,x)$ in the physical variable $x$. This result implies that the number of KL truncation terms satisfies $M\leq \lfloor  N^{\frac{1}{2s/d+1}}\rfloor$.
The third step is to approximate the eigenvalue problem of the self-adjoint operator $\mc{R}_N$
by means of a conforming Galerkin finite element method (FEM) over a regular mesh with a mesh size $h$.
Now, to estimate the error between $\kappa(y,x)$ and its numerical approximation
$\kappa_M^{N,h}(y,x)$ with $M$ being the number of truncation
terms, $N$ being the number of sampling points and $h$ being the mesh size, the eigenvalue approximation error is derived.
% in addition to the truncation error in $R_N(x,x')$, The latter was already analyzed in our recent work \cite{griebel2017decay}, and it depends on the eigenvalue decay of the covariance operator $R(x,x')$.
Moreover, to balance
the decay of the eigenvalues of the covariance kernel $R(x,x')$ and the numerical approximation
error, we need to take $h\ll N^{-1/s}$ in order to ensure convergence in the first place. Otherwise, no convergence
rate is guaranteed when solving the eigenvalue problems numerically.
%We will denote $\log\kappa_M^{N,h}$ the approximation to $\log\kappa$.

The main contribution of this work is threefold. First, we present the spectral analysis of the finite rank operator $\mathcal{R}_N$, which allows us to specify the number of truncation terms. Second, we recall the error rate of QMC quadrature and provide an error estimate of the numerical approximation to the eigenvalue problem associated with the operator $\mathcal{R}_N$ in terms of mesh size $h$. %It turns out that one should choose $h\ll N^{-1/s}$ to make the numerical approximation convergent.
Third, we derive an error estimate of both, $\log\kappa-\log\kappa_{M}^{N,h}$ and $\kappa-\kappa_{M}^{N,h}$ in various norms.
Our final estimates of $\|\kappa-\kappa_M^{N,h}\|_{L^p(\Omega,L^2( D))}$ and $\|\kappa-\kappa_M^{N,h}
\|_{L^p(\Omega, C(D))}$ with $1\leq p<2$ are presented in Theorem \ref{thm:FinalExp}. For example, we obtain the bound
\[
\norm{{\kappa - \kappa_M^{N,h} }}_ {L^p(\Omega, L^{2}(D))} \lesssim M^{-\frac{s}{d}}+M^{\frac{s}{d}+\frac{3}{2}}h^{s}+\Big(\frac{M}{N}\Big)^{1/2}.
\]
%in \eqref{eq:finalL2exp}.

%The later is
%achieved by  the error bound for the centered Gaussian random field $\log\kappa(y,x)$, the mean value theorem and Fernique's theorem.
%A critical step to establish the error estimate on $|\kappa-\kappa_M^{N,h}|$ from $|\log\kappa-\log\kappa_M^{N,h}|$

Moreover, we discuss the example of an elliptic PDE with lognormal random diffusion coefficient.
There, using our previous results on the approximation of the lognormal random field, we can deduce bounds of the error between the solution $u$ of the PDE and its induced approximation $u_M^{N,h}$.
%The error in $L^p(\Omega, H^1(D))$-norm is derived for $1\leq p<2$ in Theorem \ref{lemma:perturbation}.
%where $u$ and $u_M^{N,h}$ are solutions corresponding to $\kappa$ and $\kappa_M^{N,h}$, respectively.
%In this work, we do not consider the numerical approximation of the solutions $u$ and $u_M^{N,h}$, and
%refer to readers to \cite{char12a,gunzburger2014stochastic, Graham2015, kuo2017multilevel}
%for important progress along this line.

The remainder of the paper is organized as follows. We formulate in Section \ref{sec:prelim} the approximation
of $\log\kappa$ by the KL expansion, explain the general sampling method and discuss the Galerkin approximation. In Section \ref{sec:qmc}, we analyze the Quasi Monte-Carlo method to approximate the covariance
kernel $R(x,x')$, and derive a spectral estimate for $\mathcal{R}$ and $\mathcal{R}_N$ by means of
the maximin principle and an eigenvalue decay estimate. In Section \ref{sec:galerkin}, we discuss the
conforming Galerkin approximation of the eigenvalue problems of $\mathcal{R}_N$ and derive spectral
estimates. The main error estimates between $\kappa$ and $\kappa_M^{N,h}$ in the $L^p(\Omega,L^2(D))$-norm
and the $L^p(\Omega,C(D))$-norm with $p\in [1,2)$, respectively, are established in Section \ref{sec:main}.
Furthermore, we present an application of our results for an elliptic operator with
lognormal random coefficients in Section \ref{sec:kl_spde}. Two numerical tests are provided in
Section \ref{sec:num} to verify our findings. Finally, we give some concluding remarks in Section \ref{sec:conclusion}.

\section{Preliminaries}\label{sec:prelim}
This section collects elementary facts on the KL expansion and its numerical approximation. To this end, the overall numerical
approximation error is divided into three parts: the truncation error, the sampling error and the resulting approximation error of the eigenvalue problems.

We start with some notation. Let two Banach spaces $V_1$ and $V_2$ be given. Then, $\mathcal{B}(V_1,V_2)$ stands for the Banach space composed of all continuous linear operators from $V_1$ to $V_2$ and $\mc{B}(V_1)$ stands for $\mc{B}(V_1, V_1)$. The set of nonnegative integers is denoted by  $\mathbb{N}$. For any index $\alpha\in \mathbb{N}^d$, $|\alpha|$ is the
sum of all components. The letters $M$, $N$ and $h$ are reserved for the truncation number of the KL modes, the number of sampling points and the mesh size. We write $A\lesssim B$ if $A\leq cB$ for some absolute constant $c$ which is independent of $M$, $N$ and $h$, and we likewise write $A\gtrsim B$. Moreover, for any $m\in \mathbb{N}$, $1\leq p\leq \infty$, we follow \cite{Adam78} and define the Sobolev space $W^{m,p}(D)$ by	
$$W^{m,p}(D)=\{u\in L^{p}(D): D^{\alpha}u\in L^{p}(D) \text{ for } 0\leq|\alpha|\leq m\}.$$
It is equipped with the norm
\begin{equation*}
 \|u\|_{W^{m,p}(D)} = \left\{\begin{aligned}
 \Big(\sum\limits_{0\leq |\alpha|\leq m}\norm{D^{\alpha}u}_{L^p(D)}^{p}\Big)^{\frac{1}{p}}, & \text{ if }1\leq p<\infty,\\
 \max\limits_{0\leq |\alpha|\leq m}\norm{D^{\alpha}u}_{L^\infty(D)} ,& \text{ if } p=\infty.
\end{aligned}\right.
\end{equation*}
The space $W_{0}^{m,p}(D)$ is the closure of $C^{\infty}_{0}(D)$ in $W^{m,p}(D)$. Its dual space is $W^{-m,q}(D)$, with ${1}/{p}+{1}/{q}=1$. Also we use $H^{m}(D)=W^{m,p}(D)$ for $p=2$. $(\cdot,\cdot)$ denotes the inner product in $L^2(D)$.

\subsection{Karhunen-Lo\`{e}ve expansion: continuous level}
In this work, we consider a stochastic field $\kappa(y,x)\in L^2(\Omega\times D)$ with its logarithm being a centered Gaussian field, i.e.,
\[
\log\mathrm{d}\mc{P}({y})=\boldsymbol{\rho}\mathrm{d}{y}:=\prod_{j=1}^{d'}\rho(y_j)\mathrm{d}(y_j)\quad \text{ with } \rho(y):=\frac{1}{\sqrt{2\pi}}\exp{(-\frac{y^2}{2})},
\]
where $\mc{P}$ is the probability measure on $\Omega$ introduced in Section \ref{sec:intro}. We denote the associated integral operator $\mathcal{S}: L^{2}(D)\rightarrow  L^{2}(\Omega)$ by
\begin{align}\label{eq:S}
(\mathcal{S}v)({y})=\int_{D}\log\kappa(y,x)v(x)\mathrm{d}x,
\end{align}
whereas its adjoint operator $\mathcal{S}^{*}: L^{2}(\Omega)\rightarrow  L^{2}(D)$ is defined by
\begin{align}\label{eq:ajoint_S}
(\mathcal{S}^{*}v)(x)=\int_{\Omega}\log\kappa(y,x)v({y})\dy.
\end{align}
Let $\mathcal{R}: L^2(D)\rightarrow L^2(D)$ be defined by  $\mathcal{R}:=\mathcal{S}^{*}\mathcal{S}$.
Then $\mathcal{R}$ is a nonnegative self-adjoint Hilbert-Schmidt operator with kernel
$R\in L^2(D\times D):D\times D\to \mathbb{R}$ given by
\[
R(x,x')=\int_{\Omega}\log\kappa(y,x)\log\kappa(y,x')\dy. %:\bar{D}\times \bar{D}\rightarrow \mathbb{R}\in L^2(D\times D),
\]
This is just the covariance function of the stochastic process $\log\kappa(x,{y})$. Moreover, for any $v\in L^2(D)$, we have
\begin{equation*}
\mathcal{R}v(x)=\int_{D}R(x,x')v(x')\mathrm{d}x' = \int_D\int_\Omega \log\kappa(y,x)\log\kappa(y,x^\prime)v(x^\prime)\dy\mathrm{d}x^\prime.
\end{equation*}

The standard spectral theory for compact operators \cite{yosida78} implies that the operator $\mathcal{R}$ has
at most countably many discrete eigenvalues, with zero being the only accumulation point, and each non-zero
eigenvalue has only finite multiplicity. Let $\{\lambda_n\}_{n=1}^{\infty}$ be the sequence of eigenvalues (with multiplicity counted) associated to
$\mathcal{R}$, which are ordered nonincreasingly, and let $\{\phi_n\}_{n=1}^\infty$ be the corresponding eigenfunctions that are orthonormal in $L^2(D)$.
 Furthermore, for any $\lambda_n\neq 0$, define
\begin{equation}\label{eq:psi}
  \psi_n(y)=\frac{1}{\sqrt{\lambda_n}}\int_{D}\log\kappa(y,x)\phi_n(x)\mathrm{d}x.
  %\text{ and } \psi_n^{N}(y)=\frac{1}{\sqrt{\lambda_n^{N}}}\int_{D}\log\kappa(y,x)\phi_n^{N}(x)\mathrm{d}x.
\end{equation}
One can verify that the sequence $\{\psi_n\}_{n=1}^\infty$ is uncorrelated and
orthonormal in $L^2(\Omega)$, and therefore, $\{\psi_n\}_{n=1}^{\infty}$ are i.i.d normal random functions.
%As pointed out previously that

%However, $\{\psi_n^{N}\}_{n=1}^N$ is generally not an ONB.
Note that the sequence $\{\lambda_n\}_{n=1}^{\infty}$ can be characterized by the so-called approximation numbers (cf. \cite[Section 2.3.1]{Pietsch:1986:ES:21700}). They are defined by
\begin{align}\label{eq:a_n}
\lambda_{n}=\inf\{\norm{\mathcal{R}-L}_{\mathcal{B}(L^2(D))}:L\in \mathfrak{F}(L^2(D)), {\text{rank}}(L)< n\}
\end{align}
where $\mathfrak{F}(L^2(D))$ denotes the set of the finite rank operators on $L^2(D)$.
This equivalency is frequently employed to estimate eigenvalues by constructing finite rank approximation operators to $\mathcal{R}$.

The KL expansion of the bivariate function $\log\kappa(y,x)$ then refers to the expression
\begin{equation}\label{eq:KL}
\log\kappa(y,x)=\sum\limits_{n=1}^{\infty}\sqrt{\lambda_n}\phi_n(x)\psi_n(y),
\end{equation}
where the series converges in $L^2( \Omega\times D)$.
%For the sake of completeness, we recall a few results from \cite{griebel2017decay}.
\subsection{Karhunen-Lo\`{e}ve expansion: $M$-term truncation}
Now, we will truncate the KL expansion and discuss the resulting error. The studies on the $M$-term KL approximation to random fields are extensive. In \cite{Schwab:2006:KAR:1167051.1167057}, the authors derived the eigenvalue decay rates for random fields with their corresponding covariance kernels possessing certain regularity and considered the generalized fast multipole methods to solve the associated eigenvalue problems. Robust eigenvalue computation for smooth covariance kernels was studied in \cite{Todor2006}. A comparison of $M$-term KL truncation and the sparse grids approximation was given in \cite{Griebel.Harbrecht:2017}.

The result of this section is based on our recent paper \cite{griebel2017decay}, which proves a sharp eigenvalue decay rate under a mild assumption on the regularity of the bivariate function $\log\kappa(y,x)$ in the physical domain. To this end, we make the following assumption.
\begin{assumption}[Regularity of $\log\kappa(y,x)$]\label{A:11}
There exists some $s\geq0$ such that
$\log\kappa(y,x)\in L^{\infty}(\Omega, H^{s}(D))$.
\end{assumption}
\noindent Under Assumption \ref{A:11}, by the definition of the kernel $R(x,x^\prime)$, we have $R(x,x')\in H^{s}(D)\times H^{s}(D)$.

The following
eigenvalue decay estimate \cite[Theorems 3.2, 3.3 and 3.4]{griebel2017decay} will be used repeatedly. %, and therefore, we will list those results.
\begin{theorem}\label{thm:truncationError}
%(\cite[Theorems 3.2 and 3.3]{griebel2017decay}).
Let Assumption \ref{A:11} hold. Then, for any $M\in \mathbb{N}$ sufficiently large, there holds
\begin{align*}
{{\lambda_n}}&\approx C_{\ref{thm:truncationError}}n^{-\frac{2s}{d}-1} \text{ when $n$ is sufficiently large},\\
\Big\|{\sum\limits_{n>M}\sqrt{\lambda_n}\phi_n(x)\psi_n(y)}\Big\|_{L^2(\Omega\times D)}&\leq C_{\ref{thm:truncationError}}^{1/2}\sqrt{\frac{d}{2s}}(M+1)^{-\frac{s}{d}}.
\end{align*}
with the constant
$C_{\ref{thm:truncationError}}:=\diam{D}^{2s}C_{\rm em}(d,s)\ext\normLH{\log\kappa}{s}{D} ^2$. Here, $C_{\rm{em}}(d,s)$ denotes the embedding constant between the Lorentz sequence spaces $\ell_{\frac{d}{d+2s}, 1} \hookrightarrow \ell_{\frac{d}{d+2s}, \infty}$ and $\ext$ is a constant depending only on $D$ and $s$.
\end{theorem}

The next lemma gives the regularity of the eigenfunctions $\{\phi_n\}_{n=1}^{\infty}$.
\begin{lemma}[Regularity of the eigenfunctions $\{\phi_n\}_{n=1}^{\infty}$]
Let Assumption \ref{A:11} be valid. Then for all $0\leq \theta\leq 1$, there holds
\begin{align}\label{eq:phi_theta}
\normHp{\phi_n}{\theta s}{D}\leq C(D,d, s)n^{\frac{\theta s}{d}} \text{ when $n$ is sufficiently large.}% \quad\text{ for all }\quad 0\leq \theta\leq 1,
\end{align}
Here, $C(D,d, s)$ denotes a positive constant depending only on $D$, $d$ and $s$.
\end{lemma}
\begin{proof}
We will only prove the result for $s\in \mathbb{N}_{+}$. The case for $s\in \mathbb{R}^{+}$ can be obtained by the interpolation method.

Let $\bs{\alpha}=[\alpha_1,\cdots,\alpha_d]\in \mathbb{N}^d$ with $|\bs{\alpha}|:=\sum\limits_{i=1}^{d}\alpha_i\leq s$. The combination of Assumption \ref{A:11} and decomposition \eqref{eq:KL} and an application of Lebesgue's dominated convergence theorem lead to the expansion
\[
\partial_{x}^{\bs{\alpha}}\log\kappa(y,x)=\sum\limits_{n=1}^{\infty}\sqrt{\lambda_n}
\partial_{x}^{\bs{\alpha}}\phi_n(x)\psi_n(y).
\]
After taking the squared $L^2(\Omega\times D)$-norm on both sides, we arrive at
\[
\|\partial_{x}^{\bs{\alpha}}\log\kappa\|_{L^2(\Omega\times D)}^2=
\sum\limits_{n=1}^{\infty}\lambda_n\|\partial_{x}^{\bs{\alpha}}\phi_n\|_{L^2(D)}^2.
\]
Now we sum over all $\bs{\alpha}\in\mathbb{N}^d$ with $|\bs{\alpha}|\leq s$, and obtain by the definition of the Sobolev space $H^s(D)$ that
\begin{align*}
\|\log\kappa\|_{L^2(\Omega,H^s(D))}^2=\sum\limits_{n=1}^{\infty}\lambda_n\normHp{\phi_n}{s}{D}^2.
\end{align*}
At last, an application of Theorem \ref{thm:truncationError} gives
\begin{align}
+\infty>\sum\limits_{n=1}^{\infty}\lambda_n\normHp{\phi_n}{s}{D}^2
&\approx C_{\ref{thm:truncationError}}\sum\limits_{n=1}^{\infty}n^{-\frac{2s}{d}-1}
\normHp{\phi_n}{s}{D}^2\nonumber\\
&= C_{\ref{thm:truncationError}}\sum\limits_{n=1}^{\infty}n^{-1-\epsilon}\cdot n^{-\frac{2s}{d}+\epsilon}
\normHp{\phi_n}{s}{D}^2\label{eq:777888}
\end{align}
for any positive parameter $\epsilon$. With $0<\epsilon\to 0$ we obtain from the relation \eqref{eq:777888} that, when $n$ is sufficiently large, there holds
\begin{align*}
\normHp{\phi_n}{ s}{D}\leq C(D,d, s)n^{\frac{s}{d}}.
\end{align*}
This verifies \eqref{eq:phi_theta} for $\theta=1$. By noting that $\normL{\phi_n}{D}=1$, an application of \cite[Theorem 3.3]{Agmon65} yields then the desired estimate.
\end{proof}

It is worth to emphasize the optimality of the eigenfunctions $\{\phi_n\}_{n=1}^\infty$ in the sense that the mean-square error resulting from a finite-rank approximation of $\kappa(y,x)$ is minimized \cite{ghanem2003stochastic}. Thus, the eigenfunctions indeed minimize the truncation error in the $L^2$-sense, i.e.
\begin{align}\label{eq:opt_eigen}
\min\limits_{\substack{\{c_n(x)\}_{n=1}^{M}\subset L^2(D)\\\{c_n(x)\}_{n=1}^{M} \text{ orthonormal}}}\normL{\log\kappa(y,x)-\sum\limits_{n=1}^{M}\left(\int_{D}\log\kappa(y,x)c_n(x)\mathrm{d}x\right) c_{n}(x)}{\Omega\times D}=\sqrt{\sum\limits_{n>M}\lambda_n}.
\end{align}
\subsection{Sampling estimate of the continuous Karhunen-Lo\`{e}ve approximation}
Clearly, any numerical computation of the covariance function $R(x,x^{\prime})$ by a conventional quadrature method quickly becomes expensive and impractical when the dimensionality $d'$ of the random domain $\Omega$ is large. This is due to the curse of dimensionality.
To this end, depending on the regularity prerequisites with respect to the stochastic variable $y$, the Monte Carlo method, the Quasi-Monte Carlo (QMC) methods or the sparse grid method may be employed in approximating $R(x,x')$. In this paper, we will focus on QMC.

Anyway, a numerical quadrature gives $R_N(x,x')$, which is defined by
\begin{align}\label{eq:QMC}
  R_{N}(x,x'):=\sum\limits_{n=1}^{N}\omega_n\log\kappa({y}_n,x)\log\kappa({y}_n,x'). %\boldsymbol{\rho}({y}_n).
\end{align}
Here, $N\in \mathbb{N}$ denotes the number of quadrature points and $\{{y}_1,\cdots,{y}_N\}$ and $\{\omega_1,\cdots,\omega_N\}$ are the corresponding
quadrature points and weights. Clearly, $R_N\in L^2(D\times D)$ and $R_N: D\times D\to \mathbb{R}$.

Analogously, we denote
by $\mathcal{R}_N$ the nonnegative self-adjoint Hilbert-Schmidt operator with kernel $R_N$. The operator $\mathcal{R}_N$ is of rank no greater than $N$ and hence compact. Analogously, we can define in nondecreasing order
its eigenvalues and its normalized eigenfunctions in $L^2(D)$ as $\{\lambda_n^N\}_{n=1}^{N}$
and $\{\phi_n^N\}_{n=1}^{N}$, respectively.

%\label{rem:bivariate1}
Note at this point the following: If we are interested in a specific approximate realization of $\log\kappa(y,\cdot)$ for some $y\in \Omega$, then we have to consider the function $\psi_n^{N}(y)$ defined by
\begin{equation}\label{eq:psi_N}
\psi_n^{N}(y)=\frac{1}{\sqrt{\lambda_n^{N}}}\int_{D}\log\kappa(y,x)\phi_n^{N}(x)\mathrm{d}x.
\end{equation}
To estimate the error between $\psi_n$ and $\psi_n^{N}$, we can apply finite elements $\mathcal{T}_h$ over $D$ as introduced in Subsection \ref{subsec:fem}. This error depends on the regularity of $\log\kappa(y,\cdot)$ for given $y\in \Omega$.
On the other hand, if we are only interested in certain statistical quantities of the Gaussian random field $\log\kappa$, then there is no need to calculate $\{\psi_n^{N}\}_{n=1}^{N}$, and we can take directly i.i.d normal random functions, e.g., $\{\psi_n\}_{n=1}^{N}$. This is indeed the situation many articles are concerned with, see e.g., \cite{Bachmayr.Cohen.Dinh.Schwab:2017,char12a, kuo2017multilevel}.

\subsection{Galerkin discretization of the sampled, truncated continuous Karhunen-Lo\`{e}ve approximation}\label{subsec:fem}
Now we describe the conforming Galerkin approximation of the eigenvalue problem on
$\mathcal{R}_N$. To this end, let $\mathcal{T}_{h}$ be a regular quasi-uniform triangulation over the
physical domain $D$ with a maximal mesh size $h$ and let $k:=\lceil s \rceil$. The
associated finite element space $V_h$ is defined by
\begin{align}\label{eq:FEspace}
V_h:=\{v\in H^{1}(D):v|_{K}\in P^{k}(K) \text{ for all } K\in \mc{T}_h\}.
\end{align}
Let $Q$ be the dimension of $V_h$. We then have $Q=\mathcal{O}\Big((\frac{h}{k})^d\Big)$. The $L^2$-projection $I_h: L^2(D)\to V_h$ has the approximation property \cite[Theorem 4.4.20]{MR2373954}
\begin{align}
\normL{v-I_hv}{D}&\leq C_{I_h}h^{s}\normHp{v}{s}{D} \text{ for all } v\in H^s(D)\label{eq:approxL2}\\
h^{d/2}\normI{v-I_hv}{D}&\leq C_{I_h}h^{s}\normHp{v}{s}{D} \text{ for all } v\in H^s(D) \text{ for } s>d/2. \label{eq:approxLinfty}
\end{align}
Here, the positive constant $C_{I_h}$ depends only on the regularity parameter of $\mathcal{T}_h$ and is independent of the mesh size $h$. %[{\color{red}$P_h$ is more common for projection}]

The conforming Galerkin approximation of the eigenvalue problem of $\mathcal{R}_N$ is to find
$\{\lambda_n^{N,h},\phi_{n}^{N,h}\}_{n=1}^{Q}\subset \mathbb{R}\times V_h$ such that
\begin{align*}
I_h\Big(\lambda_n^{N,h}\phi_{n}^{N,h}-\mathcal{R}_N\phi_{n}^{N,h}\Big)=0.
\end{align*}
This is equivalent to the eigenvalue problem of the finite-rank operator on $L^2(D)$ defined by
\[
\mc{R}_{N,h}:= I_h\mc{R}_N I_h.
\]
Let $\{\lambda_n^{N,h},\phi_{n}^{N,h}\}_{n=1}^{Q}$ be the corresponding eigenpairs with eigenvalues in nonincreasing order and eigenvectors orthonormal in $L^2(D)$.
Then the $M$-term truncated KL expansion, denoted by $\kappa_M^{N,h}(y,x)$, is defined by
\begin{align}\label{eqn:KL_N}
\log\kappa_M^{N,h}(y,x):=\sum\limits_{n=1}^{M}\sqrt{\lambda_n^{N,h}}\phi_n^{N,h}(x)\psi_n(y).
\end{align}
%with
%\begin{remark}
Note at this point the following: Again, if we are mainly concerned with the approximation to a specific bivariate function $\log\kappa$ via the expression \eqref{eqn:KL_N}, then we have to replace $\psi_n$ with its numerical approximation
\begin{align}\label{eq:psireal}
{\psi}_n^{N,h}(y) :=\frac{1}{\sqrt{\lambda_n^{N,h}}}h^2\sum\limits_{n=1}^{N}\sum\limits_{K\in \mathcal{T}_h}\sum\limits_{x_j\in I_{K}}\log\kappa(x_j,y_n)\phi_n^{N,h}(x_j)L_n(y).
\end{align}
Here, $I_K$ represents the quadrature points on each finite element $K\in\mathcal{T}_h$ and
$\{L_n(y)\}_{n=1}^{N}$ denotes the Legendre polynomials of order $N$.
Note that ${\psi}_n^{N,h}(y)$ is the numerical approximation by interpolation with
sampling points $\{y_n\}_{n=1}^{N}$ to $\tilde{\psi}_n^{N,h}(y)$ defined as
\begin{align}\label{eq:psiIdeal}
\tilde{\psi}_n^{N,h}(y) :=\frac{1}{\sqrt{\lambda_n^{N,h}}}\int_{D}\log\kappa(y,x)\phi_n^{N,h}(x)\mathrm{d}x.
\end{align}
%\end{remark}
%\noteLi{How to define $\psi^{N,h}_{n}$?}
In view of the KL expansion \eqref{eq:KL} and the $M$-term truncation estimate \eqref{eqn:KL_N}, an application of the triangle inequality yields
\begin{equation}\label{TotalErr}
\begin{aligned}
\normL{\log\kappa-\log \kappa_M^{N,h}}{\Omega\times D}
&\leq
\Big\|{\sum\limits_{n>M}\sqrt{\lambda_n}\phi_n(x)\psi_n(y)}\Big\|_ {L^2(\Omega\times D)}\\
&+\Big\|{\sum\limits_{n=1}^{M}\Big(\sqrt{\lambda_n}\phi_n(x)\psi_n(y)
-\sqrt{\lambda_n^{N}}\phi_n^{N}(x)\psi_n(y)\Big)}\Big\|_ {L^2(\Omega\times D)} \\
&+
\Big\|{\sum\limits_{n=1}^{M}\Big(\sqrt{\lambda_n^{N}}\phi_n^{N}(x)\psi_n(y)
-\sqrt{\lambda_n^{N,h}}\phi_n^{N,h}(x)\psi_n(y)\Big)}\Big\|_{L^2(\Omega\times D)}.
\end{aligned}
\end{equation}
A main goal of this paper is to derive a sharp estimate of $\normL{\log\kappa-\log \kappa_M^{N,h}}{\Omega\times D}$ in \eqref{TotalErr}. To this end, it suffices to analyze the three terms on the right hand side of \eqref{TotalErr}. Here, the first term represents the truncation error that can be estimated by Theorem \ref{thm:truncationError}, the second term is due to sampling of the KL approximation and the third term is induced by the Galerkin approximation error.

%The fractional order Sobolev space $W^{s,p}(D)$, $s\geq0, s\notin \mathbb{N}$, is defined by means of interpolation \cite{Adam78}.

\section{QMC method approximation error}\label{sec:qmc}
In this section, we apply the QMC method based on the randomly shifted lattice rule and derive the sampling error corresponding to the second term in \eqref{TotalErr}. To this end, we map the quadrature points
$\Xi_N:=\{\xi_1, \xi_2, \cdots,\xi_N\}\subset [0,1]^{d'}$ to $\mathbb{R}^{d'}$ by the inverse of the cumulative distribution function of the standard normal distribution.  The cumulative distribution function $\bs\phi(y)$ is defined by
\[
\bs{\phi}({y}):=\prod_{i=1}^{d'}\phi(y_i), \text{ where }
\phi:\mathbb{R}\to (0,1) \text{ with } \phi(y_i):=\int_{-\infty}^{y_i}\rho(y')dy'
\]
and its inverse is
$\bs{\phi}^{-1}({y}): (0,1)^{d'}\to \mathbb{R}^{d'}.$
Upon changing variables, we obtain
\begin{align}\label{eq:singular}
R(x,x')=\int_{[0,1]^{d'}}\log\kappa(\bs{\phi}^{-1}(z),x)
\log\kappa(\bs{\phi}^{-1}(z),x')\mathrm{d}z.
\end{align}
Then by taking ${y}_i:=\bs{\phi}^{-1}(\xi_i)$ and $\omega_i:=\frac{1}{N}$ for $i=1,2,\cdots, N$ in \eqref{eq:QMC}, we get an approximation to
the covariance function $R(x,x')$, which is denoted by $R_N(x,x')$.

To this end, we introduce the construction of the quadrature points $\Xi_N$, which is based on the fast CBC construction of randomly
shifted lattice rules in the unanchored space \cite{Nichols2014FastCC}. The unanchored space $\mathcal{F}(\mathbb{R}^{d'})$ is defined by
\begin{align}\label{eq:mathcalF}
\mathcal{F}(\mathbb{R}^{d'})&:=\Bigg\{v\in L^2(\mathbb{R}^{d'}):\|v\|_{\mathcal{F}(\mathbb{R}^{d'})}^2:=\nonumber\\
&\sum\limits_{\bs{\alpha}\subset\{1,\cdots,d'\}}\frac{1}{\gamma_{\bs{\alpha}}}
\int_{\mathbb{R}^{|\bs{\alpha}|}}
\Big(
\int_{\mathbb{R}^{d'-|\bs{\alpha}|}}\partial_{y}^{\bs{\alpha}}
v(y_{\bs{\alpha}};y_{-\bs{\alpha}})
\rho(y_{-\bs{\alpha}})\mathrm{d}y_{-\bs{\alpha}}
\Big)^2\nu(y_{\bs{\alpha}})\mathrm{d}y_{\bs{\alpha}}<\infty
\Bigg\}.
\end{align}
Here, the positive function $\nu$ controls the boundary behavior of the functions in $\mathcal{F}(\mathbb{R}^{d\rq{}})$. The collection of parameters $\gamma_{\bs{\alpha}}$ for all $\bs{\alpha}\subset\{1,\cdots,d\rq{}\}$ controls the relative importance of various groups of variables, and
\[
\rho(y_{-\bs{\alpha}})=\Pi_{j\in \{1,d'\}\backslash\bs{\alpha}}\rho(y_j) \quad  \text{ and }\quad
\nu(y_{\bs{\alpha}})=\Pi_{j\in\bs{\alpha}}\nu(y_j) .
\]
Note that we will choose the weight function $\nu$ and the weight parameters $\gamma_{\bs{\alpha}}$, such that the bivariate function $\log\kappa(\cdot,x)$ belongs to $\mathcal{F}(\mathbb{R}^{d'})$ for all $x\in D$.

%Moreover, we denote the weight parameters as
%$
%\gamma_{\bs{\alpha}}:=(|\bs{\alpha}|!)^2\Pi_{\bs{\alpha}_j\ne 0}\frac{0.01}{j^3}.
%$
We apply the CBC approach \cite[Algorithm 6]{Nichols2014FastCC} to derive the generating vector $\bs{z}\in [0,1)^{d'}$
with the number of sampling points being $N$. To this end, let the shift $\bs{\Delta}\in [0,1]^{d'}$ be an i.i.d uniformly distributed vector. Then we obtain the randomly shifted (rank-1) lattice rule by
\begin{align}\label{eq:lattice}
\xi_i:=\frac{i\bs{z}}{N}+\bs{\Delta}-\lfloor\frac{i\bs{z}}{N}+\bs{\Delta} \rfloor, i=1,\cdots,N.
\end{align}
Now, $R_N(x,x')$ in \eqref{eq:QMC} can be approximated by taking $y_i:=\bs{\phi}^{-1}(\xi_i)$ for $i=1,2,\cdots,N$.

The error $e_{d',N}(\bs{z})$ between $R(x,x')$ and $R_N(x,x')$ is measured by the shifted-averaged worse-case error defined by
\begin{align*}
e_{d',N}(\bs{z}):=\max_{v\in \mathcal{F}(\mathbb{R}^{d'})}\Bigg\{\Big(\int_{[0,1]^{d\rq{}}}\Big|
\int_{\mathbb{R}^{d'}}v(y)\dy-\frac{1}{N}\sum\limits_{i=1}^{N}v(y_i)\Big|^2
\mathrm{d}\bs{\Delta}\Big)^{1/2}\Bigg\}.
\end{align*}
 Thus, using
the CBC Algorithm to calculate $R_N(x,x')\in L^2(D\times D)$ defined in
\eqref{eq:QMC} yields a shifted-averaged worse-case error of $\mathcal{O}(N^{-1+\delta})$ for any $\delta>0$ with the construction
cost of $\mathcal{O}(d'N\log(N))$. Therefore, we start with the following setting.

\begin{assumption}[Assumption on the sampling error]\label{ass:discrepancy}
For some $\delta\in (0,1)$, there holds
\begin{align}\label{eq:OperApproxN}
\norm{\mc{R}-\mc{R}_N}_{\mc{B}(L^2(D))}\lesssim N^{-1+\delta}.
\end{align}
\end{assumption}
%Assumption \ref{ass:discrepancy} implies
To approximate a bivariate function or a specific realization of the random field $\log\kappa(y,\cdot)$, we have introduced in the last section the quantities $\{\psi_n^{N}\}_{n=1}^{\infty}$, cf. \eqref{eq:psi_N}, which are not orthonormal in $L^2(\Omega)$. Nevertheless, they are very close to an orthonormal basis when the approximation error between $\mc{R}$ and $\mc{R}_N$ is very small.
\begin{lemma}[Near orthonormality of $\{\psi^N_{n}\}_{n=1}^{N}$]\label{lemma:orthognal}
Let ${\psi}^{N}_{n}$ be defined as in \eqref{eq:psi_N}. There holds
\[
\int_{\Omega}{\psi}^{N}_{n}{\psi}^{N}_{m}\dy=\delta_{m,n} +\frac{1}{\sqrt{\lambda^{N}_n\lambda^{N}_m}}\int_{D}(\mc{R}-\mc{R}_N)\phi^{N}_{m}\phi^{N}_{n}\mathrm{d}x\quad \text{ for all } 1\leq m,n\leq N.
\]
\end{lemma}
\begin{proof}
This is a direct consequence of the definition \eqref{eq:psi_N} and the eigenvalue problem for $\mc{R}_N$.
% after an application of the triangle inequality.
\end{proof}
Next, we give some estimates on the finite-rank approximation $\mathcal{R}_N$ and its spectrum.
%Prior to preceeding estimate to  the first term of \eqref{TotalErr}, we need recall a truncation estimate in \cite{griebel2017decay}:
\begin{proposition}[Spectral estimate for $\mc{R}_N$]\label{prop:recall}
Let Assumption \ref{A:11} hold, let $N\in \mathbb{N}_{+}$ be sufficiently large and let $M:= \lfloor  N^{\frac{1}{2s/d+1}}\rfloor$. Then $\mathcal{R}_N\in \mathcal{B}(L^2(D), H^s(D))$ with
\begin{align}\label{eq:operNnorm}
\norm{\mathcal{R}_N}_{\mathcal{B}(L^2(D), H^s(D))}
%\leq \frac{1}{N}\sum\limits_{n=1}^{N}\normL{\log\kappa(y_n,\cdot)}{ D}
%\normHp{\log\kappa(y_n,\cdot)}{s}{D}.
\leq \norm{\log\kappa}_{L^{\infty}(\Omega,L^2(D))}
\norm{\log\kappa}_{L^{\infty}(\Omega,H^s(D))}.
\end{align}
For $1\leq n\leq N$, there holds
\begin{align}\label{eq:555}
\norm{\phi_n^{N}}_{H^s(D)}&\leq(\lambda_n^{N})^{-1}\norm{\mathcal{R}_N}_{\mathcal{B}(L^2(D), H^s(D))}.
\end{align}
Furthermore, let $\lambda_{k_i}$ be an eigenvalue of $\mc{R}$ with multiplicity $q_i$ for $i=1,2,\cdots$ and $k_{I-1}< N\leq k_{I}$ for some $I\in\mathbb{N}_{+}$.
Assume that for sufficiently large $N$, there holds
\begin{align}\label{eq:spec_gap}
\frac{1}{N}\ll \min\limits_{i=2,\cdots,I}\{\lambda_{k_i}-\lambda_{k_{i-1}}\}.
\end{align}
Then, for $1\leq n\leq N$, there holds
\begin{align}
{{\lambda_n^{N}}}&\lesssim \max\{n^{-\frac{2s}{d}-1},N^{-1}\}\label{eq:specN}.
\end{align}
In addition,
\begin{align}\label{eq:eigenvectorN}
\normL{\phi_n^{N}-\phi_n}{D}\lesssim N^{-1/2} \quad\text{ for all }\quad 1\leq n \leq M.
\end{align}
\end{proposition}
\begin{proof}
We can obtain from the definition \eqref{eq:QMC} and the triangle inequality
\begin{align*}
\norm{\mathcal{R}_N}_{\mathcal{B}(L^2(D), H^s(D))}
\leq \frac{1}{N}\sum\limits_{n=1}^{N}\normL{\log\kappa(y_n,\cdot)}{ D}
\normHp{\log\kappa(y_n,\cdot)}{s}{D}.
\end{align*}
Then Assumption \ref{A:11} leads to \eqref{eq:operNnorm}. The relation \eqref{eq:555} is derived from the definition. To prove \eqref{eq:specN}, fix $1\leq n\leq N$ and let $V_n=\text{span}\{\phi_1,\cdots,\phi_n\}$ be a $n$-dimension subspace. Since $\mc{R}$ and $\mc{R}_N$ are nonnegative and self-adjoint, we obtain
\begin{align}\label{eq:444}
\lambda_n-\lambda_n^{N}\leq \lambda_n-\min\limits_{v\in V_n}\frac{(\mc{R}_Nv,v)}{(v,v)}.
\end{align}
Next we estimate the lower bound of the minimum on the right hand side of \eqref{eq:444}. To this end, note that any
$v\in V_n$ admits the expression $v=\sum\limits_{i=1}^n c_i\phi_i$ for some $\{c_i\}_{i=1}^n\subset \mathbb{R}^n$. Let $(v,v):=1$, then $\sum\limits_{i=1}^n c_i^2=1$. For any $\delta>0$, plugging in the expression for $v$ and applying \eqref{eq:OperApproxN} lead to
\begin{align}
\min\limits_{v\in V_n}\frac{(\mc{R}_Nv,v)}{(v,v)}&=\min\limits_{v\in V_n}\frac{(\mc{R}v,v)}{(v,v)}+\frac{((\mc{R}_N-\mc{R})v,v)}{(v,v)}\nonumber\\
&\geq \min\limits_{\sum\limits_{i=1}^n c_i^2=1}\sum\limits_{i=1}^{n}\lambda_{i}c_i^2-{N^{-1+\delta}}\sum\limits_{i=1}^{n}\sum\limits_{j=1}^{n} c_i c_j\nonumber\\
&=:\min\limits_{\sum\limits_{i=1}^n c_i^2=1}f(c_1,\cdots,c_n).\label{eq:333}
\end{align}
The lower bound of the minimum can now be estimated using Lagrange multipliers. To this end, let $\mu\in \mathbb{R}$ and define
\begin{align*}
F(c_1,\cdots,c_n;\mu):=f(c_1,\cdots,c_n)-\mu(\sum\limits_{i=1}^n c_i^2-1)
\end{align*}
Let $(\boldsymbol{c^*},\mu^*)=(c_1^*,\cdots,c_n^*,\mu^{*})$ be the optimal point to the unconstrained minimization problem associated to $F(c_1,\cdots,c_n;\mu)$. Then $c_1^*,\cdots,c_n^*$ have the same sign by the definition of $f$. Let $c_i^*\geq 0$ for all $1\leq i\leq n$. The optimality conditions read
$\frac{\partial F}{\partial c_i}|_{(\boldsymbol{c^*},\mu^*)}=0$  for $i=1,\cdots,n$, and $
\frac{\partial F}{\partial \mu}|_{(\boldsymbol{c^*},\mu^*)}=0$.
This immediately implies
\begin{align}
\min\limits_{\sum\limits_{i=1}^n c_i^2=1}f(c_1,\cdots,c_n)=f(\boldsymbol{c^*})&=\mu^*\quad \text{ and }\quad \forall 1\leq i\leq n: (\lambda_i-\mu^{*})c_i^*=\frac{1}{N^{1-\delta}}\sum\limits_{j=1}^{n}c_j^{*}.\label{eq:222}
%\sum\limits_{i=1}^n\frac{1}{\lambda_i-\mu^*}&=N^{1-\delta}\label{eq:111}.
\end{align}
The second relation in \eqref{eq:222} implies for all $1\leq i\leq n$ that there holds
\begin{align*}
\frac{1}{\lambda_i-\mu^{*}}=N^{1-\delta}\frac{c_i^*}{\sum\limits_{j=1}^{n}c_j^{*}}.
\end{align*}
Summing over $i=1,\cdots,n$ yields
\begin{align}
\sum\limits_{i=1}^n\frac{1}{\lambda_i-\mu^*}&=N^{1-\delta}\label{eq:111}.
\end{align}

Recall that $c_i^{*}\geq 0$ for all $i=1,\cdots,n$. Together with \eqref{eq:222}, this implies
%By \eqref{eq:OperApproxN} and \eqref{eq:222}, we can arrive at an upper bound for $\mu^{*}$, namely,
$
\mu^{*}< \lambda_n.
$
Now combining \eqref{eq:spec_gap} and \eqref{eq:111} results in
$
\frac{1}{\lambda_n-\mu^*}\gtrsim N,
$
and therefore, $\mu^*\geq \lambda_n-C_1 N^{-1}$ for some positive constant $C_1$ independent of $N$. This, together with \eqref{eq:222}, \eqref{eq:333} and \eqref{eq:444}, gives
$
\lambda_n-\lambda_n^{N}\lesssim N^{-1}.
$
Analogously, by changing the roles of $\mc{R}$ and $\mc{R}_N$, we can show
\begin{align}\label{eq:6666}
|\lambda_n^{N}-\lambda_n|\lesssim N^{-1}.
\end{align}
Consequently, \eqref{eq:specN} follows by Theorem \ref{thm:truncationError}.

It remains to prove \eqref{eq:eigenvectorN}. We only present the proof for $n=1$. For $n>1$, \eqref{eq:eigenvectorN} can
be shown similarly to \cite[Theorem 9.1]{babuska&osborn91}. Since the whole space $L^2(D)$ is orthogonally
decomposed  as the direct sum of the range of $\mc{R}$ and its kernel, $\phi_1^{N}$ can be split into
\begin{align}
\phi_1^{N}:=\sum\limits_{i=1}^{\infty}c_i \phi_i+v
\end{align}
for some $v\in L^2(D)$ satisfying $\mc{R}v=0$. Recall that $\lambda_1=\cdots\lambda_{q_1}>\lambda_{q_1+1}$. This leads to
\begin{align*}
&\big(1-\frac{\lambda_{q_1+1}}{\lambda_1}\big)\normL{\phi_1^N-\sum\limits_{i=1}^{q_1}c_i \phi_i}{D}^2  =\big(1-\frac{\lambda_{q_1+1}}{\lambda_1}\big)\Big(  \sum\limits_{n=q_1+1}^{\infty}c_n^2+\normL{v}{D}^2\Big)\\
&\leq \sum\limits_{n=1}^{\infty}\big(1-\frac{\lambda_{n}}{\lambda_1}\big)c_n^2+\normL{v}{D}^2=(\phi_1^N,\phi_1^N)-\lambda_1^{-1}(\mc{R}\phi_1^N,\phi_1^N)\\
&=\lambda_1^{-1}(\lambda_1-\lambda_1^{N})-\lambda_1^{-1}((\mc{R}-\mc{R}_N)\phi_1^N,\phi_1^N),
\end{align*}
which, combined with \eqref{eq:OperApproxN} and \eqref{eq:6666}, gives
\begin{align*}
(1-\frac{\lambda_{q_1+1}}{\lambda_1})\normL{\phi_1^N-\sum\limits_{i=1}^{q_1}c_i \phi_i}{D}^2
\lesssim \lambda_1^{-1}N^{-1}.
\end{align*}
By redefining $\phi_1$ to be
$\big(\sum\limits_{i=1}^{q_1}c_i^2\big)^{-\frac{1}{2}}\sum\limits_{i=1}^{q_1}c_i \phi_i,
$
\eqref{eq:eigenvectorN} is proved due to the spectral gap assumption \eqref{eq:spec_gap}.
\end{proof}
\begin{remark}
If the number of sampling points $N$ is not sufficiently large then the spectral gap assumption \eqref{eq:spec_gap} is not fulfilled. Then one can show that for $1\leq n\leq N$, there holds
\begin{align}
{{\lambda_n^{N}}}&\lesssim \max\{n^{-\frac{2s}{d}-1},\frac{n}{N}\}\label{eq:specN2}.
\end{align}
Thus, to make $\lambda_n^{N}$ an accurate approximation to $\lambda_n$ for $1\leq n\leq M$, we have to impose a much more stringent restriction on $M:= \lfloor  N^{\frac{1}{2s/d+2}}\rfloor$. In this manner, we can show that $\frac{n}{N}\ll n^{-\frac{2s}{d}-1}$ for $1\leq n\leq M$.
\end{remark}
\begin{comment}
Next we provide the estimate on $\|\psi_n-\psi_n^N\|_{L^2(\Omega)}$.
\begin{lemma}\label{lemma:psi_N}
Let $N$ be sufficiently large and $M:= \lfloor  N^{\frac{1}{2s/d+1}}\rfloor$. Furthermore, let \eqref{eq:spec_gap} be satisfied. Then it holds
\begin{align*}
\|\psi_n-\psi_n^N\|_{L^2(\Omega)}\lesssim \frac{1}{\sqrt{\lambda_n^{N}}} N^{-1/2}\quad\text{ for all } n=1,\cdots, M.
\end{align*}
\end{lemma}
\begin{proof}
By \eqref{eq:psi}, we can obtain
\begin{align*}
\psi_n(y)-\psi_n^{N}(y)&=\frac{1}{\sqrt{\lambda_n}}\int_{D}\log\kappa(y,x)\phi_n(x)\mathrm{d}x-\frac{1}{\sqrt{\lambda_n^{N}}}\int_{D}\log\kappa(y,x)\phi_n^{N}(x)\mathrm{d}x\\
&=\Big(\frac{1}{\sqrt{\lambda_n}}-\frac{1}{\sqrt{\lambda_n^{N}}}\Big)\int_{D}\log\kappa(y,x)\phi_n(x)\mathrm{d}x+\frac{1}{\sqrt{\lambda_n^{N}}}\int_{D}\log\kappa(y,x)(\phi_n(x)-\phi_n^{N}(x))\mathrm{d}x.
\end{align*}
Therefore, after taking $L^2(\Omega)$ norm on both sides, this yields
\begin{align*}
\|\psi_n-\psi_n^N\|_{L^2(\Omega)}&\leq \Big|1-\sqrt{\frac{{\lambda_n}}{{\lambda_n^N}}}\Big|+\frac{1}{\sqrt{\lambda_n^{N}}}\|\log\kappa\|_{L^2(\Omega\times D)}\normL{\phi_n^{N}-\phi_n}{D}\\
&\lesssim \frac{1}{\sqrt{\lambda_n^{N}}} N^{-1/2}.
\end{align*}
In the last inequality, we have applied the simple inequality that $|\sqrt{a}-\sqrt{b}|\leq \sqrt{|a-b|}$ for all $a,b\geq 0$, in addition to \eqref{eq:eigenvectorN}. This completes the proof.
\end{proof}
\end{comment}

Finally, we can give an estimate to the second term in \eqref{TotalErr}:
\begin{proposition}[Root mean square error to the second term in \eqref{TotalErr}]\label{prop:1stErr}
Let $N$ be sufficiently large and $M\leq \lfloor  N^{\frac{1}{2s/d+1}}\rfloor$. Furthermore, let \eqref{eq:spec_gap} be satisfied. Then there holds
\[
\normL{\sum\limits_{n=1}^{M}\Big(\sqrt{\lambda_n}\phi_n(x)\psi_n(y)
-\sqrt{\lambda_n^{N}}\phi_n^{N}(x)\psi_n(y)\Big)} {\Omega\times D}\lesssim M^{1/2}N^{-1/2}. %N^{-1/{2(1+\frac{d}{2s})}}.
\]
\end{proposition}
\begin{proof}
Employing the triangle inequality yields
 \begin{align*}
&\normL{\sum\limits_{n=1}^{M}\Big(\sqrt{\lambda_n}\phi_n(x)\psi_n(y)
-\sqrt{\lambda_n^{N}}\phi_n^{N}(x)\psi_n(y)\Big)} {\Omega\times D}^2\\&=
\normL{\sum\limits_{n=1}^{M}\Big((\sqrt{\lambda_n}-\sqrt{\lambda_n^N})\phi_n(x)\psi_n(y)
+\sqrt{\lambda_n^{N}}(\phi_n(x)-\phi_n^{N}(x))\psi_n(y)\Big)} {\Omega\times D}^2\\
&\leq \sum\limits_{n=1}^{M}\Big( (\sqrt{\lambda_n}-\sqrt{\lambda_n^N})^2 + \lambda_n^N\normL{\phi_n^{N}-\phi_n}{D}^2\Big).
\end{align*}
Then the desired result follows from \eqref{eq:6666} and \eqref{eq:eigenvectorN}.
\end{proof}
Next we give an estimate on $\|\psi_n-\psi_n^N\|_{L^2(\Omega)}$.
\begin{lemma}[Estimate on $\|\psi_n-\psi_n^N\|_{L^2(\Omega)}$]\label{lemma:psi_N}
Let $N$ be sufficiently large and $M\leq \lfloor  N^{\frac{1}{2s/d+1}}\rfloor$. Furthermore, let \eqref{eq:spec_gap} be satisfied. Then there holds
\begin{align*}
\|\psi_n-\psi_n^N\|_{L^2(\Omega)}\lesssim \frac{1}{\sqrt{\lambda_n^{N}}} N^{-1/2}\quad\text{ for all } n=1,\cdots, M.
\end{align*}
\end{lemma}
\begin{proof}
By \eqref{eq:psi}, we obtain
\begin{align*}
\psi_n(y)-\psi_n^{N}(y)&=\frac{1}{\sqrt{\lambda_n}}\int_{D}\log\kappa(y,x)\phi_n(x)\mathrm{d}x-\frac{1}{\sqrt{\lambda_n^{N}}}\int_{D}\log\kappa(y,x)\phi_n^{N}(x)\mathrm{d}x\\
&=\Big(\frac{1}{\sqrt{\lambda_n}}-\frac{1}{\sqrt{\lambda_n^{N}}}\Big)\int_{D}\log\kappa(y,x)\phi_n(x)\mathrm{d}x+\frac{1}{\sqrt{\lambda_n^{N}}}\int_{D}\log\kappa(y,x)(\phi_n(x)-\phi_n^{N}(x))\mathrm{d}x.
\end{align*}
Therefore, taking the $L^2(\Omega)$-norm on both sides yields
\begin{align*}
\|\psi_n-\psi_n^N\|_{L^2(\Omega)}&\leq \Big|1-\sqrt{\frac{{\lambda_n}}{{\lambda_n^N}}}\Big|+\frac{1}{\sqrt{\lambda_n^{N}}}\|\log\kappa\|_{L^2(\Omega\times D)}\normL{\phi_n^{N}-\phi_n}{D}\\
&\lesssim \frac{1}{\sqrt{\lambda_n^{N}}} N^{-1/2},
\end{align*}
where, in the last inequality, we have applied \eqref{eq:eigenvectorN} and the inequality  $|\sqrt{a}-\sqrt{b}|\leq \sqrt{|a-b|}$ for
all $a,b\geq 0$.
\end{proof}

In order to numerically approximate each realization of the Gaussian random field
$\log\kappa(y,\cdot)$ for given $y\in \Omega$, we need the following result.
\begin{proposition}\label{prop:1stErrbivariate}
Let $N$ be sufficiently large and $M\leq \lfloor  N^{\frac{1}{2s/d+1}}\rfloor$. Furthermore, let \eqref{eq:spec_gap} be satisfied. Then
\[
\normL{\sum\limits_{n=1}^{M}\Big(\sqrt{\lambda_n}\phi_n(x)\psi_n(y)
-\sqrt{\lambda_n^{N}}\phi_n^{N}(x)\psi_n^{N}(y)\Big)} {\Omega\times D}\lesssim MN^{-1/2}. %N^{-1/{2(1+\frac{d}{2s})}}.
\]
\end{proposition}
\begin{proof}
The triangle inequality yields
 \begin{align*}
&\normL{\sum\limits_{n=1}^{M}\Big(\sqrt{\lambda_n}\phi_n(x)\psi_n(y)
-\sqrt{\lambda_n^{N}}\phi_n^{N}(x)\psi_n^{N}(y)\Big)} {\Omega\times D}\\
&=
\Big\|\sum\limits_{n=1}^{M}\Big((\sqrt{\lambda_n}-\sqrt{\lambda_n^N})\phi_n(x)\psi_n(y)
+\sqrt{\lambda_n^{N}}(\phi_n(x)-\phi_n^{N}(x))\psi_n(y)\\
&\hspace{6cm}+\sqrt{\lambda_n^{N}}\phi_n^{N}(x)(\psi_n(y)-\psi_n^{N}(y))\Big)\Big\|_ {L^2(\Omega\times D)}\\
&\lesssim \sum\limits_{n=1}^{M}\Big( |\sqrt{\lambda_n}-\sqrt{\lambda_n^N}| + \sqrt{\lambda_n^N}\normL{\phi_n^{N}-\phi_n}{D}+\sqrt{\lambda_n^N}\|\psi_n-\psi_n^N\|_{L^2(\Omega)}\Big).
\end{align*}
Then the desired result follows from \eqref{eq:6666}, \eqref{eq:eigenvectorN} and Lemma \ref{lemma:psi_N}.
\end{proof}
\begin{comment}
\begin{proposition}
\[
\normL{\sum\limits_{n>M}\sqrt{\lambda_n^{N}}\phi_n^{N}(x){\psi}_n^{N}(y)} {\Omega\times D}\leq C_{\ref{thm:truncationError}}
	 (M+1)^{-\frac{s}{d}}+\sqrt{|D|\mc{D}^{r,s}(\Xi_{N})\|\log\kappa\|_{\mc{W}^{r',s'}}}.
\]
\end{proposition}
\begin{proof}
By application of Lemma \ref{lemma:orthognal} and the orthogonality of \{$\phi_{n}^{N}\}_{n=1}^{\infty}$, we obtain
\begin{align*}
&\normL{\sum\limits_{n>M}\sqrt{\lambda_n^{N}}\phi_n^{N}(x)\psi_n^{N}(y)} {\Omega\times D}=\sqrt{\sum_{n>M}\lambda_n^{N}\Big(1+\frac{1}{\lambda^{N}_n}\int_{D}(\mc{R}-\mc{R}_N)\phi^{N}_{n}\cdot\phi^{N}_{n}\mathrm{d}x\Big)}\\
&\leq\sqrt{\sum_{n>M}\lambda_n^{N}\Big(1+\frac{1}{\lambda^{N}_n}\int_{D}\Delta\mc{R}_N\phi^{N}_{n}\cdot\phi^{N}_{n}\mathrm{d}x\Big)}
\leq\sqrt{\sum_{n>M}\lambda_n^{N}}+\sqrt{\sum_{n=1}^{\infty}(\Delta\mc{R}_N\phi^{N}_{n},\phi^{N}_{n})}
\end{align*}
with $\Delta\mc{R}_N$ a positive self-adjoint operator defined on $L^2(D)$ by
\[
\Delta\mc{R}_Nv:=\int_{D}|R(x,x')-R_N(x,x')|v(x')\mathrm{d}x' \text{ for } v\in L^2(D).
\]
Whereas an application of the Parseval's identity yields
\begin{align*}
\normL{\sum\limits_{n>M}\sqrt{\lambda_n^{N}}\phi_n^{N}(x)\psi_n^{N}(y)} {\Omega\times D}
&\leq\sqrt{\sum_{n>M}\lambda_n^{N}}+\sqrt{\text{Tr}(\Delta\mc{R}_N)}\\
&=\sqrt{\sum_{n>M}\lambda_n^{N}}+\sqrt{\int_{D}|R(x,x)-R_N(x,x)|\mathrm{d}x}.
\end{align*}
The desired result follows after an application of \eqref{eq:koksma-hlawka} and Theorem \ref{thm:truncationError}.
\end{proof}
\end{comment}
%\noteLi{Therefore, $M\approx\epsilon^{-d/s}$ and $N\approx\epsilon^{-2}$}

\section{Conforming Galerkin approximation estimate}\label{sec:galerkin}
In this section we derive an estimate for the third term in \eqref{TotalErr} by means of the
approximation theory of conforming finite element methods.
To this end, let
\begin{equation}\label{eq:notationErr}
%\left\{
\begin{aligned}
\mc{E}_{N,h}&:=\mc{R}_N-\mc{R}_{N,h}\\
 e_{n}^{N,h}&:=\phi_n^{N}-\phi_n^{N,h}\\
\Delta\lambda_{n}^{N,h}&:=\lambda_n^{N}-\lambda_n^{N,h}. %Res&:=\lambda_{n}^{N,h}(\phi_{n}^{N,h},\cdot)-(\mc{R}_N\phi_{n}^{N,h},\cdot).
\end{aligned}
%\right.
\end{equation}
Then $\mc{E}_{N,h}$ is a self-adjoint
operator on $L^2(D)$ and we have the following error representation.
\begin{lemma}\label{lemma:errOper}
The error operator $\mc{E}_{N,h}$ has the property
\[
(\mc{E}_{N,h} v,v)=(v,(I-I_h)\mc{R}_{N}(I+I_h)v)\quad \text{ for all } \quad v\in L^2(D).
\]
\end{lemma}
\begin{proof}
For given $v\in L^2(D)$ and since $\mc{R}_{N,h}=I_h\mc{R}_NI_h$ and $(\mc{R}_N I_h v-I_h\mc{R}_N v, v)=0$, we obtain
\begin{align*}
(\mc{E}_{N,h} v,v)&=((\mc{R}_N-I_h\mc{R}_NI_h)v,v)+(\mc{R}_N I_h v-I_h\mc{R}_N v, v)\\
&=((I-I_h)\mc{R}_N(I+I_h)v,v).
\end{align*}
\end{proof}

%Denote $Res:=\lambda_{n}^{N,h}(\phi_{n}^{N,h},\cdot)-(\mc{R}_N\phi_{n}^{N,h},\cdot)$, there holds\footnote{These equality was learned from Carsten Carstensen during Bonn Trisemester Program}
%\begin{align*}
%\lambda_n^{N}(e_n^{N,h}, e_n^{N,h})&=\Delta\lambda_n^{N,h}+(\mc{R}_Ne_n^{N,h}, e_n^{N,h})\\
%(\mc{R}_Ne_n^{N,h}, e_n^{N,h})&=Res(e_n^{N,h})+\frac{\lambda_n^{N}+\lambda_n^{N,h}}{2}(e_n^{N,h}, e_n^{N,h})\\
%\frac{\Delta\lambda_{n}^{N,h}}{2}(e_n^{N,h},e_n^{N,h})&=\Delta\lambda_{n}^{N,h}+Res(e_n^{N,h})
%\end{align*}
%One can show that
%\[
%Res(e_n^{N,h})=-\Delta\lambda_{n}^{N,h}(\phi_n^N,\phi_n^{N,h})
%\]
A direct consequence of Lemma \ref{lemma:errOper}, together with the approximation property \eqref{eq:approxL2} and Proposition \ref{prop:recall}, is the upper bound estimate for the operator norm of $\mc{E}_{N,h}$
\begin{align}\label{eq:E_{N,h}}
\norm{\mc{E}_{N,h}}_{\mc{B}(L^2(D))}\leq 2\CI h^{s}\norm{\mc{R}_N}_{\mc{B}(L^2(D), H^s(D))}.
\end{align}
Finally, we are ready to present the main result in this section.
\begin{proposition}[Conforming Galerkin approximation estimate]\label{prop:babuska}
Let Assumption \ref{A:11} hold and let $N\in \mathbb{N}_{+}$ be sufficiently large and $M\leq \lfloor  N^{\frac{1}{2s/d+1}}\rfloor$. Assume that the spectral gap condition \eqref{eq:spec_gap} is valid.
Then there are constants $C_1$, $C_2$ and $h_0\ll N^{-1/s}$ such that
\[
\Delta\lambda_{n}^{N,h}\leq C_1(\lambda_n^{N})^{-1}h^{2s} \quad\text{ for all }\quad 0<h\leq h_0\quad \text{ and } \quad n=1,\cdots, M.
\]
Furthermore, the eigenvectors $\{\phi^{N}_{n}\}_{n=1}^{N}$ can be selected such that
\[
\normL{e^{N,h}_{n}}{D}\leq C_{2}(\lambda_{n}^{N})^{-1} h^{s}\quad\text{ for all }\quad 0<h\leq h_0\quad \text{ and } \quad n=1,\cdots, M.
\]
Here, the constants $C_1$ and $C_2$ are independent of $h$ and $N$ and $h_0>0$ is sufficiently small.
\end{proposition}
\begin{proof}
The proof follows from \cite[Theorem 9.1]{babuska&osborn91}, where the following identity plays a crucial role. We have
\begin{align*}
\lambda_n^N-(\mc{R}_N v,v)=\lambda_n^{N}(v-\phi_n^N,v-\phi_n^N)-(\mc{R}_N(v-\phi_n^N),v-\phi_n^N)
\end{align*}
 for all $v\in L^2(D)$ satisfying $(v,v)=1$.
This identity can be derived by definition directly. Together with Proposition \ref{prop:recall} and the estimate \eqref{eq:E_{N,h}}, this completes the proof.
\end{proof}
Recall that if we want to approximate a certain realization of the random field $\log\kappa(y,\cdot)$ for some $y\in \Omega$, then we have to estimate the error between
${\psi}_n^{N}$ and $\psi_n^{N,h}$. To this end, we make the following assumption.
\begin{assumption}\label{A:33}
Let ${\psi}_n^{N,h}$ and $\tilde{\psi}_n^{N,h}$ be as defined in \eqref{eq:psireal} and \eqref{eq:psiIdeal}, respectively. Assume that, for some $N_0$ sufficiently large and a sufficiently small $h_0\ll N_{0}^{-1/s}$, there holds
\[
\|\tilde{\psi}_n^{N,h}-\psi_n^{N,h}\|_{L^2(\Omega)}\lesssim N^{-1}+h^{s}\quad \text{ for all } N>N_0 \text{ and } h<h_0.
\]
\end{assumption}

\noindent Note that Assumption \ref{A:33} requires a certain regularity of the bivariate function $\log\kappa(y,x)\in L^2(\Omega\times D)$ over the stochastic domain $\Omega$ since the computation of $\tilde{\psi}_n^{N,h}$ involves the polynomial interpolation of $\psi_n^{N,h}$ over $\Omega$.

We then get the following result.
\begin{lemma}\label{lem: psiEst}
Let Assumption \ref{A:33} be valid. Let ${\psi}_n^{N}$ and $\psi_n^{N,h}$ be as defined in \eqref{eq:psi_N} and \eqref{eq:psireal}, respectively. Then, for some $N_0$ sufficiently large and a sufficiently small $h_0\ll N_{0}^{-1/s}$, there holds
\[
\|{\psi}_n^{N}-\psi_n^{N,h}\|_{L^2(\Omega)}\lesssim N^{-1}+n^{\frac{3s}{d}+\frac{3}{2}}h^{s}\quad \text{ for all } N>N_0 \text{ and } h<h_0.
\]
\end{lemma}
\begin{proof}
An application of the triangle inequality leads to
\begin{align*}
\|\psi_n^{N}-\psi_n^{N,h}\|_{L^2(\Omega)}\leq \|\psi_n^{N}-\tilde{\psi}_n^{N,h}\|_{L^2(\Omega)}+\|\tilde{\psi}_n^{N,h}-\psi_n^{N,h}\|_{L^2(\Omega)}.
\end{align*}
The second term can be estimated by Assumption \ref{A:33}. The definitions \eqref{eq:psi} and \eqref{eq:psiIdeal} imply
\begin{align*}
\psi_n^{N}(y)&-\tilde{\psi}_n^{N,h}(y)=\frac{1}{\sqrt{\lambda_n^N}}\int_{D}\log\kappa(y,x)\phi_n^{N}(x)\mathrm{d}x-\frac{1}{\sqrt{\lambda_n^{N,h}}}\int_{D}\log\kappa(y,x)\phi_n^{N,h}(x)\mathrm{d}x\\
&=\Big(\frac{1}{\sqrt{\lambda_n^{N}}}-\frac{1}{\sqrt{\lambda_n^{N,h}}}\Big)\int_{D}\log\kappa(y,x)\phi_n^{N}(x)\mathrm{d}x+\frac{1}{\sqrt{\lambda_n^{N,h}}}\int_{D}\log\kappa(y,x)(\phi_n^{N}(x)-\phi_n^{N,h}(x))\mathrm{d}x\\
&=\Big(1-\sqrt{\frac{{\lambda_n^{N}}}{{\lambda_n^{N,h}}}}\Big)\psi_n^N+\frac{1}{\sqrt{\lambda_n^{N,h}}}\int_{D}\log\kappa(y,x)(\phi_n^{N}(x)-\phi_n^{N,h}(x))\mathrm{d}x.
\end{align*}
Taking the $L^2(\Omega)$-norm on both sides leads to
\begin{align*}
\|\psi_n^{N}-{\psi}_n^{N,h}\|_{L^2(\Omega)}&\leq\Bigg|\frac{\sqrt{\lambda_n^{N}}-\sqrt{\lambda_n^{N,h}}}{\sqrt{\lambda_n^{N,h}}}\Bigg|\normL{\psi_n^N}{\Omega}+\frac{1}{\sqrt{\lambda_n^{N,h}}}\normL{\log\kappa}{\Omega\times D}\normL{e_{n}^{N,h}}{D}\\
&\leq \frac{1}{\sqrt{\lambda_n^{N,h}}}\Big( \sqrt{\Delta\lambda_{n}^{N,h}} \normL{\psi_n^N}{\Omega}+\normL{\log\kappa}{\Omega\times D}\normL{e_{n}^{N,h}}{D} \Big).
\end{align*}
%Notice that
%\[
%\Bigg|\frac{1}{\sqrt{\lambda_n^{N}}}-\frac{1}{\sqrt{\lambda_n^{N,h}}}\Bigg|=\frac{\Delta\lambda_{n}^{N,h}}%{\sqrt{\lambda_n^{N}\lambda_n^{N,h}}\Big(\sqrt{\lambda_n^{N}}+\sqrt{\lambda_n^{N,h}}\Big)},
%\]
Then the desired result follows from Lemma \ref{lemma:orthognal} and Proposition \ref{prop:babuska}.
\end{proof}

\section{Main estimates}\label{sec:main}
In this section, we present the main estimate of the error between the lognormal random field $\kappa$ and its $M$-term numerical approximation
$\kappa_{M}^{N,h}$ in the $L^p(\Omega,L^2(D))$-norm and in the $L^p(\Omega, C(D))$-norm for $p<2$. The overall procedure is as follows: We first derive in Sections
\ref{subsec:l2} and \ref{subsec:linfty} an estimate on $|\log\kappa-\log\kappa_{M}^{N,h}| $ with respect to the $L^2(\Omega\times D)$-norm and the $L^2(\Omega, C(D))$-norm, cf. Theorems \ref{thm:FinalL2} and \ref{thm:FinalL00}. Then we establish the final results on
$|\kappa-\kappa_{M}^{N,h}|$ in Theorem \ref{thm:FinalExp} by employing Fernique's Theorem.

\subsection{$L^2$ error estimate}\label{subsec:l2}

First, we give an estimate for the third term in \eqref{TotalErr}.
\begin{proposition}[Galerkin approximation estimate in \eqref{TotalErr}]\label{prop:2ndErr}
Let Assumption \ref{A:11} hold, let $N\in \mathbb{N}_{+}$ be sufficiently large and let $M\leq \lfloor  N^{\frac{1}{2s/d+1}}\rfloor$ and $h\ll N^{-1/s}$. Assume the spectral gap condition \eqref{eq:spec_gap} to be valid. Then there holds
\begin{align*}
&\normL{\sum\limits_{n=1}^{M}\Big(\sqrt{\lambda_n^{N}}\phi_n^{N}\psi_n
-\sqrt{\lambda_n^{N,h}}\phi_n^{N,h}\psi_n\Big)} {\Omega\times D}\lesssim M^{\frac{s}{d}+\frac{3}{2}}h^{s}.
\end{align*}
\end{proposition}
\begin{proof}
Due to the orthogonality of the basis functions $\{\psi_n\}_{n=1}^{\infty}$ in $L^2(\Omega)$, an application of the triangle inequality leads to
\begin{align*}
&\normL{\sum\limits_{n=1}^{M}\Big(\sqrt{\lambda_n^{N}}\phi_n^{N}\psi_n
-\sqrt{\lambda_n^{N,h}}\phi_n^{N,h}\psi_n\Big)} {\Omega\times D}
=\normL{\sum\limits_{n=1}^{M}\Big(\sqrt{\lambda_n^{N}}\phi_n^{N}
-\sqrt{\lambda_n^{N,h}}\phi_n^{N,h}\Big)}{D}\\
&= \normL{\sum\limits_{n=1}^{M}\Big((\sqrt{\lambda_n^{N}}-\sqrt{\lambda_n^{N,h}})\phi_n^{N}
+\sqrt{\lambda_n^{N,h}}(\phi_n^{N}-\phi_n^{N,h})\Big)} {D}\\
&\leq \sqrt{\sum\limits_{n=1}^{M}\Delta\lambda_{n}^{N,h} }+\sum\limits_{n=1}^{M}\sqrt{\lambda_n^{N,h}}\normL{e_{n}^{N,h}}{D}.
\end{align*}
Here, we have used the orthogonality of $\{\phi_n^{N}\}_{n=1}^{M}$ over $L^2(D)$ in the last step.
%By Lemma \ref{lemma:orthognal} and \eqref{eq:OperApproxN},
%\[
%\|\psi_n^{N}\|_{L^2(\Omega)}\lesssim 1 \quad \text{ for all } 1\leq n\leq M.
%\]
Then, an application of Proposition \ref{prop:babuska} and Theorem \ref{thm:truncationError} gives the desired result.
\end{proof}

Now, using Theorem \ref{thm:truncationError}, Propositions \ref{prop:1stErr} and \ref{prop:2ndErr}, we are finally ready to present an estimate for \eqref{TotalErr}.
\begin{theorem}[Root mean square error for $M$-term KL expansion]\label{thm:FinalL2}
Let Assumption \ref{A:11} hold. Let $N\in \mathbb{N}_{+}$ be large, let $M\leq \lfloor  N^{\frac{1}{2s/d+1}}\rfloor$ and $h\ll N^{-1/s}$. Assume the spectral gap condition \eqref{eq:spec_gap} to be valid. Let $\log\kappa_{M}^{N,h}$ be given as in \eqref{eqn:KL_N}.
Then there holds
\begin{align*}
\normL{\log\kappa-\log \kappa_M^{N,h}}{\Omega\times D} \lesssim M^{-\frac{s}{d}}+\Big(\frac{M}{N}\Big)^{1/2}+M^{\frac{s}{d}+\frac{3}{2}}h^{s}.
\end{align*}
\end{theorem}
\begin{remark}[Complexity]\label{rm:complex}
According to Theorem \ref{thm:FinalL2}, in order to approximate the Gaussian random field $\log\kappa$ by formula \eqref{eqn:KL_N} with root mean square error of $\epsilon$ for a certain threshold accuracy $\epsilon>0$, we need to take the number of sampling points $N:=\mathcal{O}(\epsilon^{-2-d/s})$, the number of truncation terms $M:=\mathcal{O}(\epsilon^{-d/s})$ and the mesh size $h:=\mathcal{O}(\epsilon^{2/s+3d/2s^2})$.
\end{remark}
%\noteLi{Therefore, the optimal rate can be given by $h\approx M^{-4/d-2/s}\approx N^{-2/s}$. In that case, the error is $N^{-\frac{1}{2+d/s}}$}
To numerically approximate the realization of $\log\kappa(y,\cdot)$ for given $y\in \Omega$ by the $M$-term truncation formula \eqref{eqn:KL_N}, we can replace the i.i.d normal random functions $\{\psi_n(y)\}_{n=1}^{M}$ with $\{\psi_n^{N,h}(y)\}_{n=1}^{M}$ defined in \eqref{eq:psireal}. The error in this process can be estimated as follows.
\begin{theorem}[Root mean square error for $M$-term KL expansion of a bivariate function]\label{thm:FinalL2bivariate}
Let Assumptions \ref{A:11} and \ref{A:33} hold. Let $N\in \mathbb{N}_{+}$ be large, let $M\leq \lfloor  N^{\frac{1}{2s/d+1}}\rfloor$ and $h\ll N^{-1/s}$. Assume the spectral gap condition \eqref{eq:spec_gap} to be valid.  Let
\begin{align}\label{eq:bivariate}
\log\kappa_M^{N,h}(y,x)=\sum\limits_{n=1}^{M}
\sqrt{\lambda_n^{N,h}}\phi_n^{N,h}(x)\psi_n^{N,h}(y).
\end{align}
Then there holds
\begin{align*}
\normL{{\log\kappa(y,x)-\log \kappa_M^{N,h}(y,x)}}{\Omega\times D} \lesssim M^{-\frac{s}{d}}
+MN^{-1/2}
+M^{\frac{2s}{d}+\frac{3}{2}}h^{s}.
\end{align*}
\end{theorem}
\begin{proof}
An application of the triangle inequality together with the KL expansion \eqref{eq:KL} implies
\begin{align*}
\normL{{\log\kappa-\log \kappa_M^{N,h}}}{\Omega\times D}&\leq
\normL{\sum\limits_{n=1}^{M}\Big(\sqrt{\lambda_n}\phi_n(x)\psi_n(y)
-\sqrt{\lambda_n^{N}}\phi_n^{N}(x)\psi_n^{N}(y)\Big)} {\Omega\times D}\\
&+
\normL{\sum\limits_{n=1}^{M}\Big(\sqrt{\lambda_n^{N}}\phi_n^{N}(x)\psi_n^{N}(y)
-\sqrt{\lambda_n^{N,h}}\phi_n^{N,h}(x)\psi_n^{N,h}(y)\Big)} {\Omega\times D}\\
&+\normL{\sum\limits_{n>M}\sqrt{\lambda_n}\phi_n(x)\psi_n(y)} {\Omega\times D}.
\end{align*}
Now, the first term and the third term above can be bounded by Proposition \ref{prop:1stErrbivariate} and
Theorem \ref{thm:truncationError}, respectively. We only need to estimate the second term.
The triangle inequality gives
\begin{align*}
&\normL{\sum\limits_{n=1}^{M}\Big(\sqrt{\lambda_n^{N}}\phi_n^{N}\psi_n^{N}
-\sqrt{\lambda_n^{N,h}}\phi_n^{N,h}\psi_n^{N,h}\Big)} {\Omega\times D}\\
&= \normL{\sum\limits_{n=1}^{M}\Big((\sqrt{\lambda_n^{N}}-\sqrt{\lambda_n^{N,h}})\phi_n^{N}\psi_n^{N}
+\sqrt{\lambda_n^{N,h}}(\phi_n^{N}-\phi_n^{N,h})\psi_n^{N}+\sqrt{\lambda_n^{N,h}}\phi_n^{N,h}(\psi_n^{N}-\psi_n^{N,h})\Big)} {\Omega\times D}\\
&\leq \sqrt{\sum\limits_{n=1}^{M}\Delta\lambda_{n}^{N,h} \|\psi_n^{N}\|_{L^2(\Omega)}^2}+\sum\limits_{n=1}^{M}\sqrt{\lambda_n^{N,h}}\normL{e_{n}^{N,h}}{D}\|\psi_n^{N}\|_{L^2(\Omega)}
+\sqrt{\sum\limits_{n=1}^{M}\lambda_n^{N,h}\|\psi_n^{N}-\psi_n^{N,h}\|_{L^2(\Omega)}^2}.
\end{align*}
By Lemma \ref{lemma:orthognal} and \eqref{eq:OperApproxN},
$
\|\psi_n^{N}\|_{L^2(\Omega)}\lesssim 1$ for all $1\leq n\leq M.
$
Then, Proposition \ref{prop:babuska}, Lemma \ref{lem: psiEst} and Theorem \ref{thm:truncationError} together show the desired result.
\end{proof}
\begin{remark}[Complexity]
According to Theorem \ref{thm:FinalL2bivariate}, in order to approximate a specific realization of the Gaussian random field $\log\kappa$ by formula \eqref{eq:bivariate} with root mean square error for a certain threshold accuracy $\epsilon>0$, we need to choose the number of sampling points $N:=\mathcal{O}(\epsilon^{-2-2d/s})$, the number of truncation terms $M:=\mathcal{O}(\epsilon^{-d/s})$ and the mesh size $h:=\mathcal{O}(\epsilon^{3/s+3d/2s^2})$.
\end{remark}
\subsection{Uniform error estimate}\label{subsec:linfty}
In order to derive a uniform error estimate of the Gaussian random field $\log\kappa(y,x)$, we require a further regularity assumption on $\log\kappa$ to guarantee that $\log\kappa\in L^2(\Omega,C(D))$. To this end, we make the following assumption.
\begin{assumption}[Regularity of $\log\kappa(y,x)$]\label{A:22}
 Let Assumption \ref{A:11} hold. Furthermore, assume $s>d/2$.
\end{assumption}
\noindent Then, the following estimate is valid.
\begin{theorem}[Uniform estimate on the eigenfunctions]\label{the:uniformB}
Let Assumption \ref{A:22} be satisfied. Then there holds
\begin{align}\label{eq:infty1}
\|\phi_n\|_{C(D)}\leq C(D,d, s) n^{\frac{1}{2}}.
\end{align}
%Furthermore, let $s>d/2+1$. Then an application of \eqref{eq:phi_theta} with $\theta s>d/2+1$ together with the Sobolev %embedding implies
%\begin{align}\label{eq:infty1nabla}
%\|\nabla\phi_n\|_{C(D)}\leq C(D,d, s) n^{\frac{1}{2}+\frac{1}{d}}.
%\end{align}
\end{theorem}
\begin{proof}
Due to Assumption \ref{A:22}, an application of \eqref{eq:phi_theta} with $\theta s>d/2$ together with the Sobolev embedding implies the desired result.
\end{proof}
\begin{remark}[Optimality of the uniform estimate \eqref{eq:infty1}]
In \cite[Section 3]{Bachmayr.Cohen.Migliorati} the uniform estimate of eigenfunctions was studied under certain assumptions on the stationary covariance kernel. There, the authors derived a similar uniform estimate as \eqref{eq:infty1} by utilizing essentially the regularity of the Fourier transform of the covariance kernel. They showed the sharpness of their uniform estimate in the case when $D=[0,1]$ and for the stationary covariance kernel $R(x,x')=R(x-x')$ with its Fourier transform $\hat{R}=\chi_{[-F,F]}$. Then, the uniform estimate of the n-th eigenfunction is $\mathcal{O}(n^{1/2})$, see \cite{bonami2016uniform}.
\end{remark}
\begin{proposition}[Uniform truncation estimate]\label{prop:L003rd}
Let Assumption \ref{A:22} be satisfied. Then, for any $1\leq M\in \mathbb{N}$, there holds
\[
\Big\|{\sum\limits_{n>M}\sqrt{\lambda_n}\phi_n(x)\psi_n(y)}\Big\|_ {L^2(\Omega,C(D))}\lesssim M^{-\frac{s}{d}+\frac{1}{2}}.
\]
\end{proposition}
\begin{proof}
By the Sobolev embedding theorem, Assumption \ref{A:22} implies $R(x,x')\in C(D\times D)$. Then one can obtain
\begin{align*}
R(x,x)=\sum\limits_{n=1}^{\infty}\lambda_n|\phi_n(x)|^2\quad\mbox{and}\quad
\|\log\kappa-\log\kappa_M\|^2_{L^2(\Omega,C({D}))}
%=\sup\limits_{x\in D}\sum\limits_{n>M}\{\lambda_n|\phi_n|^2\}
\leq \sum\limits_{n>M}\lambda_n\|{\phi_n}\|_{C(D)}^2. % \label{eq:indentity}
\end{align*}
This and Theorem \ref{thm:truncationError} yield the desired estimate.
\end{proof}
\begin{comment}
\begin{proposition}
\[
\Big\|{\sum\limits_{n>M}\sqrt{\lambda_n^{N}}\phi_n^{N}(x)\psi_n^{N}(y)}\Big\|_ {L^2(\Omega,C (D))}\leq C(D,d, s) C_{\ref{thm:truncationError}}
	 (M+1)^{-\frac{s}{d}}+C(D,d, s)\sqrt{|D|\mc{D}^{r,s}(\Xi_{N})\|\log\kappa\|_{\mc{W}^{r',s'}}}.
\]
\end{proposition}
\begin{proof}
By application of Lemma \ref{lemma:orthognal} and the orthogonality of \{$\phi_{n}^{N}\}_{n=1}^{\infty}$, we obtain
\begin{align*}
&\Big\|{\sum\limits_{n>M}\sqrt{\lambda_n^{N}}\phi_n^{N}(x)\psi_n^{N}(y)}\Big\|_ {L^2(\Omega,C (D))}\\
&\leq\sqrt{\sum_{n,m>M}\lambda_n^{N}
\|\phi_n^N\|_{C(D)}\|\phi_m^N\|_{C(D)}\Big(\delta_{m,n}+\frac{1}{\lambda^{N}_n}\int_{D}(\mc{R}-\mc{R}_N)\phi^{N}_{n}\cdot\phi^{N}_{n}\mathrm{d}x\Big)}\\
%&\leq\sqrt{\sum_{n>M}\lambda_n^{N}\Big(1+\frac{1}{\lambda^{N}_n}\int_{D}\Delta\mc{R}_N\phi^{N}_{n}\cdot\phi^{N}_{n}\mathrm{d}x\Big)}
&\leq\sqrt{\sum_{n>M}\lambda_n^{N}\|\phi_n^N\|_{C(D)}^2}+\sqrt{\sum_{n,m>M}^{\infty}\|\phi_n^N\|_{C(D)}\|\phi_m^N\|_{C(D)}(\Delta\mc{R}_N\phi^{N}_{n},\phi^{N}_{m})}
\end{align*}
with $\Delta\mc{R}_N$ a positive self-adjoint operator defined on $L^2(D)$ by
\[
\Delta\mc{R}_Nv:=\int_{D}|R(x,x')-R_N(x,x')|v(x')\mathrm{d}x' \text{ for } v\in L^2(D).
\]
The desired result follows after an application of \eqref{eq:koksma-hlawka}, Theorem \ref{thm:truncationError} and Theorem \ref{the:uniformB}.
\end{proof}
\end{comment}
\begin{proposition}\label{prop:inftyOper}
Let Assumption \ref{A:22} be fulfilled. Then $\mathcal{R}$ and $\mathcal{R}_N$ are bounded from $L^2(D)$ to  $C(D)$, i.e., $\mathcal{R}$ and $\mathcal{R}_N\in\mathcal{B}(L^2(D),C(D))$. In addition, it holds
\begin{align}
\norm{\mathcal{R}}_{\mathcal{B}(L^2(D), C(D))}&\leq \|\log\kappa\|_{L^{2}(\Omega,C(D))}\|\log\kappa\|_{L^{2}(\Omega\times D)},\label{eq:Rinf}\\
\norm{\mathcal{R}_N}_{\mathcal{B}(L^2(D), C(D))}
&\leq \frac{1}{2N}\sum\limits_{n=1}^{N}\Big(\normL{\log\kappa(y_n,\cdot)}{ D}^2+\|{\log\kappa(y_n,\cdot)}\|_{C(D)}^2\Big),\label{eq:RhInf}\\
\norm{\mathcal{E}_{N,h}}_{\mathcal{B}(H^{s}(D), C(D))}
&\leq \CI\norm{\mathcal{R}_N}_{\mathcal{B}(L^2(D), H^{s}(D))} h^{s-\frac{d}{2}}. \label{eq:ErrOperInf}
\end{align}
Furthermore, the eigenfunctions $\phi_n^{N}\in H^s(D)\hookrightarrow  C(D)$. Let $N\in \mathbb{N}_{+}$ be sufficiently large and $M:= \lfloor  N^{\frac{1}{2s/d+1}}\rfloor$. Assume the spectral gap condition \eqref{eq:spec_gap} to be valid. When $0<h\leq h_0\ll N^{-\frac{1}{s}}$ for some sufficiently small $h_0$, there holds
\begin{align}\label{eq:666}
\normI{e_n^{N,h}}{D}\lesssim \Big( h^{-\frac{d}{2}}+(\lambda_n^{N})^{-1} \Big)(\lambda_n^{N})^{-1} h^{s} \quad \text{ for all }\quad n=1,\cdots,M.
\end{align}
\end{proposition}
\begin{proof}
The proof of \eqref{eq:Rinf} and \eqref{eq:RhInf} follows directly from basic operator theory. The bound \eqref{eq:ErrOperInf} is a result of \eqref{eq:approxLinfty} and \eqref{eq:RhInf}.
By the definitions of $\phi_n^N$ and $\phi_n^{N,h}$, we obtain
\begin{align*}
e_n^{N,h}&=(\lambda_n^{N})^{-1}\mc{R}_N \phi_n^{N}-(\lambda_n^{N,h})^{-1}\mc{R}_{N,h} \phi_n^{N,h}\\
&=(\lambda_n^{N})^{-1}\mathcal{E}_{N,h} \phi_n^{N}+\Big((\lambda_n^{N})^{-1}-(\lambda_n^{N,h})^{-1}\Big)\mc{R}_{N,h} \phi_n^{N}+(\lambda_n^{N,h})^{-1}\mc{R}_{N,h} e_n^{N,h}.
\end{align*}
Together with Proposition \ref{prop:babuska} and \eqref{eq:ErrOperInf}, this yields
\begin{align*}
\normI{e_n^{N,h}}{D}\lesssim \Big( h^{-\frac{d}{2}}+(\lambda_n^{N})^{-2} h^{s}+(\lambda_n^{N,h})^{-1} \Big)(\lambda_n^{N})^{-1} h^{s}.
\end{align*}
Since $h\ll N^{-\frac{1}{s}}$, the second term can be bounded from above by the third term, and this completes the proof.
\end{proof}
\begin{remark}
In a similar manner as in the proof to \eqref{eq:666}, if $N\in \mathbb{N}_{+}$ is sufficiently large and $M:= \lfloor  N^{\frac{1}{2s/d+1}}\rfloor$, one can show that
\begin{align}\label{eq:777}
\normI{\phi_n-\phi_n^{N}}{D}\lesssim \lambda_n^{-1}{N^{-1/2}} \quad \text{ for all }\quad n=1,\cdots,M.
\end{align}
\end{remark}
Combining \eqref{eq:777} and Theorem \ref{prop:inftyOper}, analogously to Proposition \ref{prop:recall}, we can derive the following estimate.
\begin{proposition}\label{prop:L001st}
Let $N$ be sufficiently large and $M\leq \lfloor  N^{\frac{1}{2s/d+1}}\rfloor$. Furthermore, let \eqref{eq:spec_gap} be satisfied. Then there holds
\begin{align*}
&\norm{\sum\limits_{n=1}^{M}\Big(\sqrt{\lambda_n}\phi_n\psi_n
-\sqrt{\lambda_n^{N}}\phi_n^{N}\psi_n\Big)}_ {L^2(\Omega, C(D))} \lesssim M^{\frac{s}{d}+1}N^{-1/2}.
\end{align*}
\end{proposition}
\begin{proof}
This result follows from an application of the triangle inequality and \eqref{eq:777}.
\end{proof}
\begin{proposition}\label{prop:L002nd}
Let $N$ be sufficiently large and $M\leq \lfloor  N^{\frac{1}{2s/d+1}}\rfloor$. There holds
\begin{align*}
&\norm{\sum\limits_{n=1}^{M}\Big(\sqrt{\lambda_n^{N}}\phi_n^{N}\psi_n
-\sqrt{\lambda_n^{N,h}}\phi_n^{N,h}\psi_n\Big)}_ {L^2(\Omega, C(D))}\lesssim M^{\frac{3s}{d}+2}N^{-1/2}h^{s}+M^{\frac{s}{d}+\frac{3}{2}}h^{s}.
\end{align*}
\end{proposition}
\begin{proof}
An application of the triangle inequality leads to
\begin{align*}
&\norm{\sum\limits_{n=1}^{M}\Big(\sqrt{\lambda_n^{N}}\phi_n^{N}\psi_n
-\sqrt{\lambda_n^{N,h}}\phi_n^{N,h}\psi_n\Big)} _ {L^2(\Omega, C(D))}\\
&= \norm{\sum\limits_{n=1}^{M}\Big((\sqrt{\lambda_n^{N}}-\sqrt{\lambda_n^{N,h}})\phi_n^{N}\psi_n
+\sqrt{\lambda_n^{N,h}}(\phi_n^{N}-\phi_n^{N,h})\psi_n\Big)}_ {L^2(\Omega, C(D))}\\
&\leq \sqrt{\sum\limits_{n=1}^{M}\normI{\phi_n^N}{D}^2 \Delta\lambda_n^{N,h}}
+\sqrt{\sum\limits_{n=1}^{M}\lambda_n^{N,h}\normI{e_{n}^{N,h}}{D}^2}.
\end{align*}
Then, an application of the inequalities \eqref{eq:infty1} and \eqref{eq:777}, Propositions \ref{prop:babuska} and \ref{prop:inftyOper} and Theorem \ref{thm:truncationError} reveals the desired result.
\end{proof}
Finally, the uniform estimate between $\log\kappa$ and $\log\kappa_{M}^{N,h}$ can be derived from Propositions \ref{prop:L001st}, \ref{prop:L002nd} and \ref{prop:L003rd}. We then obtain the following result.
\begin{theorem}[Uniform estimate on $M$-term KL truncation of $\log\kappa$]\label{thm:FinalL00}
Let Assumption \ref{A:22} hold and let $N\in \mathbb{N}_{+}$ be large and $M\leq \lfloor  N^{\frac{1}{2s/d+1}}\rfloor$. Assume the spectral gap condition \eqref{eq:spec_gap} to be valid.
Then there exists $h_0\ll N^{-\frac{1}{s}}$ sufficiently small, such that
\begin{align*}
\norm{{\log\kappa(y,x)-\log \kappa_M^{N,h}(y,x)}}_ {L^2(\Omega, C(D))} &\lesssim M^{-\frac{s}{d}+\frac{1}{2}}+M^{\frac{3s}{d}+2}N^{-1/2}h^{s}+ M^{\frac{s}{d}+\frac{3}{2}}h^{s} \\ &+M^{\frac{s}{d}+1}N^{-1/2}.
\end{align*}
{ for all }$\quad 0<h\leq h_0$.
\end{theorem}
\subsection{Numerical estimate for the error between $\kappa$ and $\kappa_{M}^{N,h}$}
%This section is only applicable when $\kappa(y,x)$ is regarded as a centered lognormal random field.
In this section, by utilizing the preceding results on $|\log\kappa-\log\kappa_M^{N,h}|$ together with the
mean value theorem, we will derive an error estimate between $\kappa$ and $\kappa_{M}^{N,h}$. Note at this point that the
results in this part can be only applied to the case when $\log\kappa$ is a normal random field. One
crucial tool which we will employ repeatedly below is Fernique's theorem. For convenience,
we recall it in the following.
\begin{theorem}[Fernique's theorem]
Let $E$ be a real, separable Banach space and suppose that $X$ is an $E$-valued random variable which is a centered and Gaussian in the sense that, for each $x^*\in E^*$, $\langle X, x^*\rangle$ is a centered, $\mathbb{R}$-valued Gaussian random variable. If $R= \inf\Big\{r\in [0,\infty): \mathbb{P}(\|X\|_E\leq r)\geq \frac{3}{4}\Big\}$, then
\begin{align*}
\int_{\Omega} \exp{\Big(\frac{\|X\|_{E}^2}{18R^2}\Big)}\dy \lesssim 1.
\end{align*}
\end{theorem}

First, we give a priori bounds on $\kappa$ and $\kappa_M^{N,h}$.
\begin{proposition}\label{prop:exp_bounded}
Let Assumption \ref{A:22} hold and let $N\in \mathbb{N}_{+}$ be large and $M\leq \lfloor  N^{\frac{1}{2s/d+1}}\rfloor$. Assume the spectral gap condition \eqref{eq:spec_gap} to be valid.
Then there exists $h_0\ll N^{-\frac{1}{s}}$ sufficiently small such that, { for all }$\quad 0<h\leq h_0$, there holds
\begin{align*}
\forall 0<p<\infty: \quad \|{\kappa}\|_{L^p(\Omega, C(D))}\lesssim 1 \quad\text{ and }\quad  \|{\kappa_M^{N,h}}\|_{L^p(\Omega, C(D))}\lesssim 1.
\end{align*}
\end{proposition}
\begin{proof}
%The first estimate can be obtained from \cite[Propositions 2.3]{char12a}. Notice that $\|\kappa_M^{N,h}(\cdot,y)\|_{C(D)}$ is a Gaussian random variable on $\Omega$. Therefore, the second estimate can be derived in a similar manner by means of the Fernique Theorem.
Note that $\log\kappa$ is a symmetric Gaussian random variable defined on $\Omega$ and valued in $C(D)$.
By Fernique's theorem, there exists $\alpha>0$ such that
\begin{align}\label{eq:Fernique}
\int_{\Omega} \exp{\Big(\alpha\|\log\kappa(\cdot,y)\|^2_{C(D)}\Big)}\dy \lesssim 1.% \quad \text{ and }\int_{\Omega} \exp{\Big(\alpha\|\log\kappa_M^{N,h}(\cdot,y)\|^2_{C(D)}\Big)}\dy\lesssim 1.
\end{align}
Hence, by Young's inequality, we obtain
\begin{align*}
\int_{\Omega} \|\kappa(\cdot,y)\|^p_{C(D)}\dy &=\int_{\Omega} \exp{\Big(p\|\log\kappa(\cdot,y)\|_{C(D)}\Big)}\dy\\
&\leq \int_{\Omega} \exp{\Big(\alpha\|\log\kappa(\cdot,y)\|^2_{C(D)}+\frac{p^2}{4\alpha}\Big)}\dy,
\end{align*}
and \eqref{eq:Fernique} leads to
\begin{align*}
\int_{\Omega} \|\kappa(\cdot,y)\|^p_{C(D)}\dy
\lesssim \exp{(\frac{p^2}{4\alpha})}.
\end{align*}
This shows the first assertion. The second one can be obtained in a similar manner.
\end{proof}

Now we can state the main result of this section.
\begin{theorem}\label{thm:FinalExp}
Let Assumption \ref{A:22} hold. Let $N\in \mathbb{N}_{+}$ be sufficiently large, let $M\leq \lfloor  N^{\frac{1}{2s/d+1}}\rfloor$ and $h\ll N^{-\frac{1}{s}}$. Assume the spectral gap condition \eqref{eq:spec_gap} to be valid. Then, for all $ p<2$, there holds
\begin{align}
\norm{{\kappa - \kappa_M^{N,h} }}_ {L^p(\Omega, C(D))} &\lesssim M^{-\frac{s}{d}+\frac{1}{2}}+M^{\frac{3s}{d}+2}N^{-1/2}h^{s}+ M^{\frac{s}{d}+\frac{3}{2}}h^{s}+M^{\frac{s}{d}+1}N^{-1/2}\label{eq:finalLinftyexp}\\
\norm{{\kappa - \kappa_M^{N,h} }}_ {L^p(\Omega, L^{2}(D))} &\lesssim M^{-\frac{s}{d}}+M^{\frac{s}{d}+\frac{3}{2}}h^{s}+\Big(\frac{M}{N}\Big)^{1/2}\label{eq:finalL2exp} .
\end{align}
\end{theorem}
\begin{proof}
The mean value theorem indicates
\[
\forall x,y\in \mathbb{R}:\quad |\mathrm{e}^{x}-\mathrm{e}^y|\leq |x-y|(\mathrm{e}^{x}+\mathrm{e}^y).
\]
This, combined with H\"{o}lder's inequality, leads to
\begin{align*}
&\norm{{\kappa - \kappa_M^{N,h} }}_ {L^p(\Omega, C(D))}\\
%=\int_{\Omega}\normI{\mathrm{e}^{\log\kappa(y,\cdot)} - \mathrm{e}^{\log\kappa_M^{N,h}(y,\cdot)}}{D}^p\mathrm{d}\rho(y)\\
%&\leq \Big(\int_{\Omega} {\normI{\log\kappa(y,\cdot) - \log\kappa_M^{N,h}(y,\cdot)}{D}^2}\mathrm{d}\rho(y)\Big)^{1/2}\Big( \int_{\Omega}\normI{\kappa(y,\cdot)}{D}^q  \mathrm{d}\rho(y)
%+  \int_{\Omega}\normI{\kappa_M^{N,h}(y,\cdot)}{D}^q  \mathrm{d}\rho(y) \Big)^{1/q}\\
&\leq \norm{{\log\kappa(y,x)-\log \kappa_M^{N,h}(y,x)}}_ {L^2(\Omega, C(D))}\Big( \|{\kappa}\|_{L^q(\Omega, C(D))} +\|{\kappa_M^{N,h}}\|_{L^q(\Omega, C(D))}\Big),
\end{align*}
where $1/p=1/2+1/q$. In view of Theorem \ref{thm:FinalL00} and Proposition \ref{prop:exp_bounded}, this proves \eqref{eq:finalLinftyexp}. The
second assertion \eqref{eq:finalL2exp} can be shown similarly using Theorem \ref{thm:FinalL2}. This completes the proof.
\end{proof}
\section{Application to elliptic PDEs with random diffusion coefficient}\label{sec:kl_spde}
In this section, we use the results of Theorem \ref{thm:FinalExp} to
analyze a model order reduction algorithm for a class of elliptic PDEs with lognormal random coefficient in the multi-query context. In the algorithm, we apply the Karhunen-Lo\`{e}ve approximation
to the stochastic diffusion coefficient $\kappa(y,x)$ to arrive at a truncated model with finite-dimensional
noise. We shall provide an error analysis below.
Throughout this section, we assume that the conditions of Theorem \ref{thm:FinalExp} are satisfied.

Let $D$ be an open bounded domain in $\mathbb{R}^d$ with a strong local Lipchitz boundary
and let $(\Omega,\Sigma, \mathcal{P})$ be a given probability space. Consider the elliptic PDE with random coefficient
\begin{equation}\label{eqn:spde}
\left\{  \begin{aligned}
  \mathcal{L} u(y,\cdot)&=f,\quad x\in D,\\
     u(y,\cdot)&=0,\quad x\in \partial D,
  \end{aligned}\right.
\end{equation}
for a.e. $y\in\Omega$, where the elliptic operator $\mathcal{L}$ is defined by
\begin{equation*}
 \mathcal{L} u(y,\cdot) =-\nabla\cdot(\kappa(y,x)\nabla u(y,x)),
\end{equation*}
and $\nabla$ denotes the derivative with respect to the spatial variable $x$.
We assume the force term $f$ to be in $H^{-1}(D)$. In the model problem \eqref{eqn:spde}, the dependence of the diffusion coefficient
$\kappa(y,x)$ on a stochastic variable $y\in \Omega$ reflects imprecise knowledge or lack of information.

The extra-coordinate $y$ poses significant computational challenges. One popular approach is the stochastic Galerkin method \cite{babuska2004galerkin}. There, one often
approximates the stochastic diffusion coefficient $\kappa(y,x)$ by a finite sum of products of
deterministic and stochastic orthogonal {bases} (with respect to a certain probability measure).
This gives a computationally more tractable finite-dimensional noise model. There, the choice of the employed orthogonal {basis} is
crucial for the accurate and efficient approximation to $\kappa(y,x)$.
In this work, we consider the KL approximation $\kappa_M^{N,h}(y,x)$ of the random
field $\kappa(y,x)$ in \eqref{eqn:KL_N}.

First, we specify the functional analytic setting. Let $V=H^{1}_{0}(D)$ and let
$H^{-1}(D)$ be its dual space. Then, for any given $y\in \Omega$, the weak formulation of problem \eqref{eqn:spde} is to find $u(y,\cdot)\in V$ such that
\begin{align}\label{eqn:weakform}
\int_{D}\kappa(y,x)\nabla u(y,x)\cdot\nabla v(x)\mathrm{d}x=\int_{D}f(x)v(x)\mathrm{d}x\quad \forall v\in V.
\end{align}
We first discuss the well-posedness of problem \eqref{eqn:weakform} for each $y\in \Omega$, which was proven in \cite[Theorem 2.2]{char12a}. By Assumption \ref{A:22}, $\kappa(y,x)\in C(D)$ a.e.. Let $\kappa_{\text{min}}(y):=\min\limits_{x\in \bar{D}}\kappa(y,x)$ and $\kappa_{\text{max}}(y):=\max\limits_{x\in \bar{D}}\kappa(y,x)$ for all $y\in \Omega$.
\begin{proposition}[Ellipticity and boundedness of $\kappa$]\label{A:1} The following bounds hold:
\[
\forall\; 0<p<\infty: \quad\|\kappa_{\text{min}}^{-1}\|_{L^p(\Omega)}\lesssim 1 \quad \text{ and }\quad \|\kappa_{\text{max}}\|_{L^p(\Omega)}\lesssim 1.
\]
\end{proposition}

Proposition \ref{A:1}, together with the Lax-Milgram theorem, guarantees that the weak formulation \eqref{eqn:weakform} is well-posed. Furthermore,
\begin{align}\label{eqn:uniformBound}
  \|u\|_{L^p(\Omega,H^1(D))}\lesssim {\normHNg{f}{D}}  \quad \text{ for all } p>0.
\end{align}
%Thus $\mathcal{L}:V\rightarrow H^{-1}(D)$ is an invertible linear operator with inverse $\mathcal{S}=\mathcal{L}^{-1}: H^{-1}(D)\to V$ that both depend on the stochastic diffusion coefficient $\kappa(y,x)$.
%Clearly, $\mathcal{S}$ is a self-adjoint operator for all $y\in \Omega$.
%We shall denote by $U_\kappa=\{\mathcal{S}(y)f,\forall y\in \Omega \}\subset V$ the solution manifold.

After substituting the numerical KL approximation $\kappa_M^{N,h}(y,x)$ of the diffusion coefficient $\kappa(y,x)$ into problem \eqref{eqn:spde},
we arrive at a truncated problem with finite-dimensional noise: For a.e. $y\in\Omega$
\begin{equation}\label{eqn:spde_KL}
\left\{\begin{aligned}
\mathcal{L}_{M}^{N,h} u_M^{N,h}(y,\cdot)& = f,\quad x\in D,\\
u_M^{N,h}(y,\cdot)&=0,\quad x\in \partial D,
\end{aligned}\right.
\end{equation}
where $\mathcal{L}_M^{N,h}$ is the elliptic differential operator with the diffusion coefficient $\kappa_M^{N,h}$.
The corresponding weak formulation is then to find $u_M^{N,h}(y,\cdot)\in V$ such that
\begin{align}\label{eqn:weakform_M}
\int_{D}\kappa_M^{N,h}(y,x)\nabla u_M^{N,h}(y,x)\cdot\nabla v(x)\mathrm{d}x=\int_{D}f(x)v(x)\mathrm{d}x\quad \forall v\in V,
\end{align}
for any given $y\in \Omega$. Analogous to the continuous case,
let $\kappa_{\text{min}}^{M,N,h}(y):=\min\limits_{x\in \bar{D}}\kappa_{M}^{N,h}(y,x)$ and $\kappa_{\text{max}}^{M,N,h}(y):=\max\limits_{x\in \bar{D}}\kappa_M^{N,h}(y,x)$ for all $y\in \Omega$. We can then state the well-posedness of problem \eqref{eqn:weakform_M}.
\begin{proposition}[Ellipticity and boundedness of $\kappa_M^{N,h}$]\label{A:2.1} The following bounds hold:
\[
\forall \; 0<p<\infty: \quad\|(\kappa_{\text{min}}^{M,N,h})^{-1}\|_{L^p(\Omega)}\lesssim 1 \quad \text{ and }\quad \|\kappa_{\text{max}}^{M,N,h}\|_{L^p(\Omega)}\lesssim 1.
\]
\end{proposition}
\begin{proof}
This follows from the proof of \cite[Theorem 2.2]{char12a}.
\end{proof}

Due to Proposition \ref{A:2.1} and the Lax-Milgram theorem, we obtain the well-posedness of problem \eqref{eqn:weakform_M}. Furthermore, \eqref{eqn:weakform_M} and Proposition \ref{A:2.1}, together with Poincar\`{e}'s inequality, give the following {\em a priori} estimate
\begin{align}\label{eq:apriori}
\|u_M^{N,h}\|_{L^p(\Omega,H^1(D))}\lesssim \|{f}\|_{H^{-1}(D)} \quad \text{ for all } p>0.
\end{align}

The next result quantifies the {effect} of the perturbation of the coefficient $\kappa(y,x)$ on the solution $u(y,x)$.
\begin{theorem}\label{lemma:perturbation}
Let Assumption \ref{A:22} hold. Let $N\in \mathbb{N}_{+}$ be large, let $M\leq \lfloor  N^{\frac{1}{2s/d+1}}\rfloor$ and $h\ll N^{-\frac{1}{s}}$. Assume the spectral gap condition \eqref{eq:spec_gap} to be valid. Let $u$ and $u_M^{N,h}$ be solutions to \eqref{eqn:spde} and \eqref{eqn:spde_KL}, respectively.
Then for all $p<2$, there holds
\begin{align*}
\|{u(y,\cdot)-u_M^{N,h}(y,\cdot)}\|_{L^p(\Omega, H^1(D))}\lesssim\Big(M^{-\frac{s}{d}+\frac{1}{2}}+M^{\frac{3s}{d}+2}N^{-1/2}h^{s}+ M^{\frac{s}{d}+\frac{3}{2}}h^{s}+M^{\frac{s}{d}+1}N^{-1/2}\Big)\norm{f}_{H^{-1}(D)}.
\end{align*}
\end{theorem}
\begin{proof}
From the weak formulations for $u(y,x)$ and $u_M^{N,h}(y,x)$, cf. \eqref{eqn:weakform} and \eqref{eqn:weakform_M}, we obtain for any $y\in\Omega$
\begin{align}\label{eqn:1}
\int_{D}\kappa(y,x)&\nabla (u(y,x)-u_M^{N,h}(y,x))\cdot\nabla v(x)\mathrm{d}x\\
&=\int_{D}(\kappa_M^{N,h}(y,x)-\kappa(y,x))\nabla u_M^{N,h}(y,x)\cdot\nabla v(x)\mathrm{d}x\quad \forall v\in V.\nonumber
\end{align}
By setting $v=u-u_M^{N,h}\in V$ in the weak formulation \eqref{eqn:1} and using Proposition \ref{A:1} and the generalized H\"{o}lder inequality, we have
\begin{align*}
&\kappa_{\text{min}}(y)\normHsemi{u(y,\cdot)-u_M^{N,h}(y,\cdot)}{D}^2\leq \int_{D}\kappa(y,x)|\nabla (u(y,x)-u_M^{N,h}(y,x))|^2\mathrm{d}x\\
&=\int_{D}(\kappa_M^{N,h}(y,x)-\kappa(y,x))\nabla u_M^{N,h}(y,x)\cdot\nabla (u(y,x)-u_M^{N,h}(y,x))\mathrm{d}x\\
&\leq\norm{\kappa_M^{N,h}(y,\cdot)-\kappa(y,\cdot)}_{C(D)}
\normHsemi{u(y,\cdot)-u_M^{N,h}(y,\cdot)}{D}
\norm{\nabla u_M^{N,h}(y,\cdot)}_{L^{2}(D)}.
\end{align*}
Consequently, we arrive at
\begin{align*}
\normHsemi{u(y,\cdot)-u_M^{N,h}(y,\cdot)}{D}\leq \kappa_{\text{min}}^{-1}(y)
\norm{\kappa_M^{N,h}(y,\cdot)-\kappa(y,\cdot)}_{C(D)}
\norm{\nabla u_M^{N,h}(y,\cdot)}_{L^{2}(D)}.
\end{align*}
Finally, taking the $L^p(\Omega)$-norm on both sides and employing the generalized H\"{o}lder's inequality, combined with Theorem \ref{thm:FinalExp}, Proposition \ref{A:1} and the {\em a priori }estimate \eqref{eq:apriori}, shows the desired result.
\end{proof}

\begin{remark}
Note that this work is mainly concerned with the numerical approximation of the lognormal random coefficient. Therefore,
we refrain from discussing the important issue of the numerical approximation of the associated elliptic problems, i.e., Problems
\eqref{eqn:spde} and \eqref{eqn:spde_KL}. We refer to \cite{kuo2017multilevel} for related results in this direction.
\end{remark}
\begin{comment}
Let us finally give a remark to justify Assumption \ref{A:3}.
To this end, recall that the embedding
$$ W^{1,p_2}_{0}(D)\hookrightarrow H^1_{0}(D)$$
is {compact} if $p_2>2$.[{\color{red}this is not compact!}]
Thus Assumption \ref{A:3} ensures the compactness of $U_\kappa$ in $V$, which guarantees that the solution space $U_{\kappa}$
actually lies in a low-dimensional manifold despite the infinite dimensionality of the random coefficient $\kappa(y,x)$.
The compactness is clearly also valid for the approximate solution space $U_{\kappa_M}$ which is induced by the KL truncation $\kappa_M(y,x)$.
This observation naturally motivates the aim to identify the effective manifold to enable a fast approximation,
which is especially important in the multi-query context, e.g., for optimal control and Bayesian inversion \cite{SchwabGittelson:2011,CohenDeVore:2015}.
\end{comment}

\section{Numerical simulation}\label{sec:num}
In this section, we provide numerical tests to verify the theoretical results presented in Section \ref{sec:main}.
Recall that the three parameters $M$, $N$ and $h$ denote the number of terms in the KL approximation, the number of
sampling points and the mesh size. These parameters determine directly the computational cost involved.

We take $\Omega:=\mathbb{R}^{d'}$ in the following simulation. In order to obtain the $M$-term KL truncation estimate
to $\log\kappa(y,x)$ in the form of \eqref{eqn:KL_N}, we employ the fast CBC construction of randomly
shifted lattice rules in the unanchored space \cite{Nichols2014FastCC} to estimate the kernel $\mathcal{R}_N(x,x')
\in L^2(D\times D)$ defined in \eqref{eq:QMC}. To this end, we employ the unanchored space $\mathcal{F}(\mathbb{R}^{d'})$ as defined by \eqref{eq:mathcalF}, where the weight function and the weight parameters are
\begin{equation*}
\left\{
\begin{aligned}
\psi(y_j)&:=1 &&\text{for all } j=1,\cdots,d'\\
\gamma_{\bs{\alpha}}&:=(|\bs{\alpha}|!)^2\Pi_{j\in\bs{\alpha}}\frac{0.01}{j^3} &&\text{ for all } \bs{\alpha}\subset\{1,\cdots,d'\}.
\end{aligned}
\right.
\end{equation*}

We apply the CBC method \cite[Algorithm 6]{Nichols2014FastCC} to derive the generating vector $\bs{z}\in [0,1)^{d'}$
with the number of sampling points being $N=1009$. To this end, let the shift $\bs{\Delta}\in [0,1]^{d'}$ be an i.i.d uniformly distributed vector. Then we obtain the randomly shifted (rank-1) lattice rule by formula \eqref{eq:lattice}.  
%\begin{align}\label{eq:lattice}
%\xi_i:=\frac{i\bs{z}}{N}+\bs{\Delta}-\lfloor\frac{i\bs{z}}{N}+\bs{\Delta} \rfloor, i=1,\cdots,N.
%\end{align}
Now, $R_N(x,x')$ in \eqref{eq:QMC} can be approximated by taking $y_i:=\bs{\phi}^{-1}(\xi_i)$ for $i=1,2,\cdots,N$.

The bivariate functions $\log\kappa(\cdot,x)$ employed in the following examples belong to $\mathcal{F}(\mathbb{R}^{d'})$ for all $x\in D$. Thus, using
the CBC Algorithm to calculate $\mathcal{R}_N(x,x')\in L^2(D\times D)$ as defined in
\eqref{eq:QMC} yields a shifted-averaged worse-case error of $\mathcal{O}(N^{-1+\delta})$ for any $\delta>0$ with the construction
cost of $\mathcal{O}(d'N\log(N))$.
\subsection{Example 1: $d'=1$ and $N=1009$}
Let
\begin{align}\label{eq:eg1}
\log\kappa(y,x):=e^{-|x-y|} \quad \text{ with } x\in [0,1] \text{ and } y\in \mathbb{R},
\end{align}
see Fig. \ref{fig:eg1} for an illustration.
One can verify that $\log\kappa(y,x)\in L^2(\Omega,H^{3/2-\delta}(D))$ for any $\delta>0$ with the physical domain
$D:=[0,1]$ and the stochastic domain $\Omega:=\mathbb{R}$. Thus, we
have $s:=3/2-\delta$ in this case. According to the definition of the finite element space $V_h$, we will use conforming quadratic finite element. Now as in Remark \ref{rm:complex}, let the tolerance be chosen as $\epsilon:=0.1$. Note that we always fix the sampling points $N=1009$. Then we can take the number of truncation terms $M\approx 5$ and the mesh size $h\approx 1/101$.

\begin{figure}[H]
  \centering
  \includegraphics[width=0.5\textwidth]{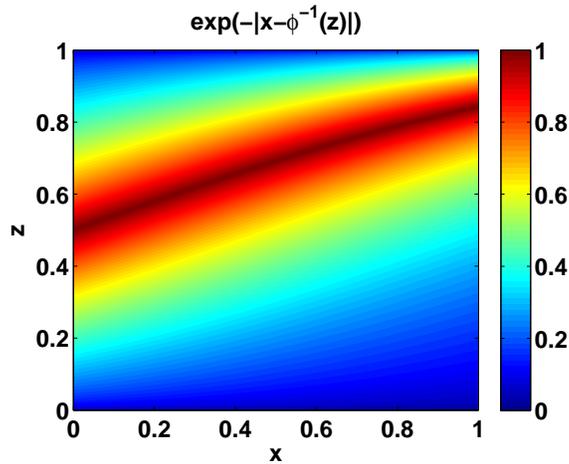}
  \caption{An illustration of the bivariate function in \eqref{eq:eg1}.}
  \label{fig:eg1}
\end{figure}

Since the dimension of the stochastic domain equals to one, we can choose the generating vector $\bs{z}:=1$. We then can derive the sampling points $\{y_i\}_{i=1}^{N}:=\{\bs{\phi}^{-1}(\xi_i)\}_{i=1}^{N}$ from formula \eqref{eq:lattice}. The shifted-averaged worse-case error is 8.0171e-6.
We present in Table \ref{table:re3} the root mean square error between $\log\kappa$ and $\log\kappa_{M}^{N,h}$ for different numbers of
truncation terms $M=2,4,8$ and different mesh sizes $h:=1/16,1/64, 1/128, 1/256$.

\begin{table}[H]
\begin{center}
  \caption{The root mean square error between $\log\kappa$ and $\log\kappa_{M}^{N,h}$ for different numbers of
truncation terms $M$ and different mesh sizes $h$. Here, the number of sampling points is $N=1009$ and the dimension $d'=1$ with the optimal parameters being $(M,h):=(5,1/101)$ and the corresponding root mean square error being 8.0493e-3.}
  \label{table:re3}

  \begin{tabular}{ | c | c | c | c |}
    \hline
    $h\backslash M$ & 2 & 4 & 8 \\ \hline
    1/16& 4.1217e-2 & 1.1933e-2 &3.6194e-3\\ \hline
    1/64 & 4.1216e-2 & 1.1931e-2 &3.6060e-3\\\hline
    1/128 & 4.1216e-2 & 1.1931e-2 &3.6059e-3\\\hline
    1/256 & 4.1216e-2 & 1.1931e-2 &3.6059e-3\\\hline
  \end{tabular}
\end{center}
\end{table}
Now, let us compare these computed results with the values that were
predicted from our theory. To this end, we plug the fixed number of sampling points $N=1009$ into Remark \ref{rm:complex}
and derive that we can take the accuracy $\epsilon:=0.1$, the number of truncation terms $M:=5$ and the mesh size $h:=1/101$. Indeed, for $(M,h):=(5,1/101)$, we also obtain the optimal error in Table \ref{table:re3}. This shows that our estimates are quite sharp and involve just small constants.

\subsection{Example 2: $d'=10,100$ and $N=1009$}
Let the bivariate function $\kappa(y,x)$ be given by
\begin{align}\label{eq:eg2}
\log\kappa(y,x):=e^{-|x-1/2|\times\|y\|_{\ell_1}} \quad \text{ with } x\in [0,1] \text{ and } y\in \mathbb{R}^{d'}.
\end{align}
Then one can verify that $\log\kappa(y,x)\in L^2(D,H^{t}(\Omega))$ for any $t>0$, with the physical domain $D:=[0,1]$
and the stochastic domain $\Omega:=\mathbb{R}^{d'}$. According to the definition of the finite element space $V_h$, we will use spectral element up to degree 10. The basis functions are Lagrange interpolation
polynomials through the local Gauss-Lobatto integration points defined per element. Due to Theorem \ref{thm:truncationError},
the eigenvalues decay very fast since $s=\infty$ in this case.
%Analogous to the previous test, one can show that the optimal number of the truncation terms $M:=5$ and the mesh size $h:=1/101$.

We present the root mean square errors between $\log\kappa$ and $\log\kappa_{M}^{N,h}$ in Tables \ref{table:re1} and \ref{table:re2} for different mesh sizes $h:=1/16, 1/64,1/128 \text{ and }1/256$ and different numbers of truncation terms $M:=2,4 \text{ and }8$ with $d'=10$ and $d'=100$, respectively.
%Our theoretical results imply the selection of $h=N^{-1/1.5}\approx 1/101$ and $M:=N^{-1/4}\approx 5$, the combination of these two parameters yields the mean square errors of $2.2803e-3$ and $1.834e-2$.

\begin{table}[h]
\begin{center}
  \caption{The root mean square error between $\log\kappa$ and $\log\kappa_{M}^{N,h}$ for different numbers of
truncation terms $M$ and different mesh sizes $h$. Here, the number of sampling points is $N=1009$, dimension $d'=10$ and the shifted-averaged worse-case error is 3.0987e-3.}
  \label{table:re1}
  \begin{tabular}{ | c | c | c | c |}
    \hline
    $h\backslash M$ & 2 & 4 & 8 \\ \hline
    1/16& 6.0723e-3 & 1.8676e-4 &1.7078e-4\\ \hline
    1/64 & 7.2336e-3 & 1.0267e-4 &1.0837e-5\\\hline
    1/128 & 6.2803e-3 & 7.9009e-5 &2.7146e-6\\\hline
    1/256 & 6.1652e-3 & 6.6978e-5 &2.7102e-6\\\hline
  \end{tabular}
\end{center}
\end{table}

\begin{table}[H]
\begin{center}
  \caption{The root mean square error between $\log\kappa$ and $\log\kappa_{M}^{N,h}$ for different numbers of
truncation terms $M$ and different mesh sizes $h$. Here, the number of sampling points $N=1009$, dimension $d'=100$ and the shifted-averaged worse-case error is 3.1045e-3.}
  \label{table:re2}
  \begin{tabular}{ | c | c | c | c |}
    \hline
    $h\backslash M$ & 2 & 4 & 8 \\ \hline
    1/16& 2.8446e-3 & 2.8332e-3 &2.8335e-3\\ \hline
    1/64 & 3.7890e-4 & 2.9770e-4 &2.9743e-4\\\hline
    1/128 & 2.8832e-4 & 7.8182e-5 &7.8268e-5\\\hline
    1/256 & 2.2743e-4 & 1.9821e-5 &1.9772e-5\\\hline
  \end{tabular}
\end{center}
\end{table}
Furthermore, for our fixed number of sampling points $N:=1009$ and for the accuracy $\epsilon:=0.1$, we can compare our computed results with the predicted ones due to Remark \ref{rm:complex}. We see that our estimates are again qualitatively quite sharp and involve just small constants.

%As predicted by Remark \ref{rm:complex}, the optimal parameters $(M,h)$ are of $\mathcal{O}(1)$ to achieve this accuracy. This agrees with the numerical tests in Tables \ref{table:re1} and \ref{table:re2}.

\section{Concluding remarks}\label{sec:conclusion}
In this work, we have analyzed the numerical approximation error in the Karhunen-Lo\`{e}ve expansion to log normal
random coefficients. We derived the numerical error in terms of the number $M$ of terms in the
Karhunen-Lo\`{e}ve expansion, the number $N$ of QMC sampling points to estimate the covariance function and
the mesh size $h$ for the conforming Galerkin approximation to the eigenvalue problem. Our results show the basic
relation
$$M\leq N^{\frac{1}{2s/d+1}} \text{ and } h\ll N^{-1/s}$$
among those three parameters, where $d$ is the dimension of the physical domain and $s$ denotes the regularity of the bivariate function in the physical domain. These results are also useful for the study of stochastic elliptic problems. We presented numerical results for one and multiple stochastic dimensions to support our theory.

The QMC method can be replaced by some properly adapted sparse grid method, if there is higher mixed regularity in $\log\kappa$ present with respect to the stochastic variables. Analogously, if the physical problem possesses higher regularity, then a more suitable FEM of higher order can be utilized. Then, of course, the sampling estimate, the Galerkin estimate and the resulting error estimates have to be modified accordingly. This would lead to a different balancing of the terms that in Remark \ref{rm:complex}.

\section*{Acknowledgements}
The authors were supported by the Hausdorff Center for Mathematics in Bonn and the Sonderforschungsbereich 1060 {\em The Mathematics of Emergent Effects} funded by the Deutsche Forschungsgemeinschaft. Guanglian Li acknowledges the support from the Royal Society via the Newton International fellowship. Part of this work was done during her visit to IPAM in the Long program: Computational Issues in Oil Field Applications.
We thank Prof. Dr. Markus Bachmayr for fruitful discussions.
\bibliographystyle{abbrv}
\bibliography{reference}

\end{document}